%% file: wdgp_siamuq.tex
\documentclass[final,onefignum,onetabnum]{siamonline250211}

\input{ex_shared}
\ifpdf
\hypersetup{
  pdftitle={Wavelet compression for sampling deep Gaussian processes},
  pdfauthor={Oskar Broukal, Michael Feischl, Conor Osborne}
}
\fi

\begin{document}

\maketitle

\begin{abstract}
  Deep Gaussian processes (DGPs) use iterative sampling to create non-stationary random fields.  This is notoriously expensive because the covariance at each non-initial layer depends on the realised field from the layer below, so it changes with every sample. Moreover, standard fast sampling techniques like circulant embedding or matrix compression fail since the underlying kernel usually lacks the necessary regularity and structure. We develop and analyse a wavelet compression method for convolution Mat\'ern DGPs. At each layer the covariance matrix is compressed using wavelets and a Lanczos algorithm  computes the random sample. For specific parameter ranges involving the smoothness of the Mat\'ern kernel as well as the smoothness of the wavelets, we show that the accumulative compression and sampling errors through all layers decay exponentially with increasing truncation parameter. We obtain a practical DGP sampling algorithm by choosing a truncation that gives log-linear covariance storage and matrix–vector costs, enabling DGP sampling at nearly linear computational cost.
\end{abstract}



\section{Introduction}\label{sec:introduction}

Deep Gaussian processes (DGPs) compose conditionally Gaussian random fields to produce hierarchical non-Gaussian fields, in which the covariance at each layer depends on the realised field from the layer below \cite{damianou2013deep,dunlop2018deep}. This flexibility comes with a substantial computational cost. For a Gaussian process with a fixed covariance matrix, the cost of a single factorisation is amortised over repeated samples; in a DGP the covariance at layer \(\ell\) depends on the realised field at layer \(\ell-1\), so it changes as a sample is propagated through the hierarchy and a single factorisation cannot in general be reused. At a spatial resolution with \(p\) degrees of freedom, dense factorisation of a covariance matrix requires \(O(p^3)\) arithmetic and \(O(p^2)\) storage \cite{Rasmussen06gaussianprocesses}, so repeated simulation becomes prohibitive as the resolution grows.
Additionally, standard fast sampling techniques such as circulant embedding~\cite{circulant_embedding1,circulant_embedding2} and matrix compression methods~\cite{Feischl2018} do not work because the involved kernels are usually not structured or regular enough.

\medskip

In the following, we show that a standard wavelet compression is enough to overcome these limitations and can produce approximate samples of a DGP in nearly linear computational cost in \(p\), with high-probability error bounds that decay exponentially with the truncation level at fixed spatial resolution. The main difficulty of the analysis is to track the approximation and compression errors through subsequent layers of the DGP. We underline the theoretical results with numerical experiments.

\medskip

A classical approach to wavelet compression of integral operators is developed in \cite{BeylkinCoifmanRokhlin1991}, where decay of the kernel away from the diagonal and the cancellation properties of the wavelet basis yield entrywise decay estimates from which most entries may be discarded; the resulting compression theory and its complexity are developed further in \cite{Stevenson2004,Dahmen2006}. For Gaussian random fields, \cite{Harbrecht2024} constructs sparse multilevel representations of covariance and precision operators, including the Mat\'ern class, and obtains fast sampling algorithms by exploiting wavelet compressibility and operator structure; this is the closest existing work to the covariance-compression component of the present paper. Multilevel techniques for the fast simulation of Gaussian random fields are also developed in \cite{HerrmannKirchnerSchwab2020}. A related matrix-function viewpoint appears in \cite{Feischl2018}, where an \(\mathcal H^2\)-matrix approximation of a covariance matrix is combined with a Krylov/Lanczos procedure for the square-root action \(\bC^{1/2}z\) without forming a dense square root. The present work shares the principle of combining covariance compression with iterative matrix-function actions, but both the compression mechanism and the matrix function arising after compression are different.

Covariances depending on a parameter have been treated by reduced-order and SPDE-based methods. Reduced-basis surrogates of Karhunen--Lo\`{e}ve eigenspaces for parameterised covariance operators are constructed in \cite{LatzEisenbergerUllmann2019}, while \cite{LatzTeckentrupUrbainczyk2025} develops computational methods for DGP priors from fractional-SPDE representations and rational approximation. Mathematical analysis of non-stationary DGP priors is developed in \cite{OsborneTeckentrup2026}, and the convolution Mat\'ern hierarchy used here belongs to the broader class of non-stationary constructions based on input warping, kernel mixtures and spatially varying correlation lengths \cite{Sampson1992,Fuentes2001,Paciorek2003,Volodina2020,noack2024unifying}.

\medskip

Our approach considers the finite-dimensional covariance matrices in a biorthogonal multiresolution wavelet basis directly. For a covariance matrix \(\bC\) indexed by wavelets at levels \(j\) and \(j'\), we consider a truncation $\bC^L$ that retains every diagonal entry as well as an off-diagonal entry only when \(j+j'\le L\). The resulting sparsity structure is illustrated in Figure~\ref{fig:sparsity-patterns}.
For each layer, standard wavelet compression arguments show an exponential decay of compression error with respect to $L$. Therefore, sufficiently large $L$ also guarantees that the truncated matrix is positive semidefinite. However, numerical experiments show such a choice is often prohibitively large in practice. For this reason, we include a projection of the truncated matrix onto its positive spectral part $\widehat{\bC}^{\,L}:=(\bC^L)_+$ into the error analysis and use the remaining matrix for sampling.
\begin{figure}[htbp]
    \centering

    \begin{subfigure}[t]{0.32\textwidth}
        \centering
        \includegraphics[width=\textwidth]{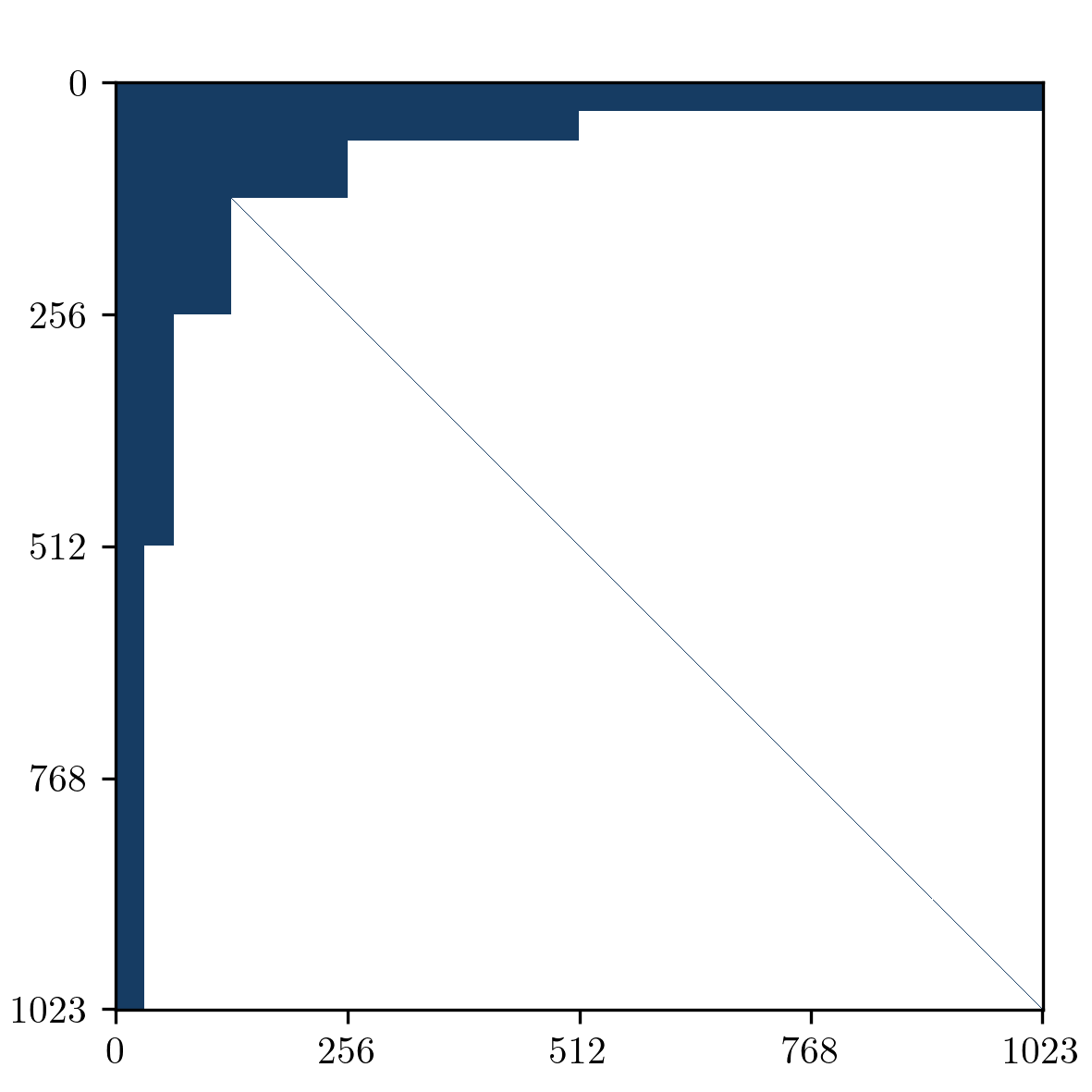}
        \caption{\(11.02\%\) non-zeros at \(L=13\)}
        \label{fig:sparsity-L13}
    \end{subfigure}
    \hfill
    \begin{subfigure}[t]{0.32\textwidth}
        \centering
        \includegraphics[width=\textwidth]{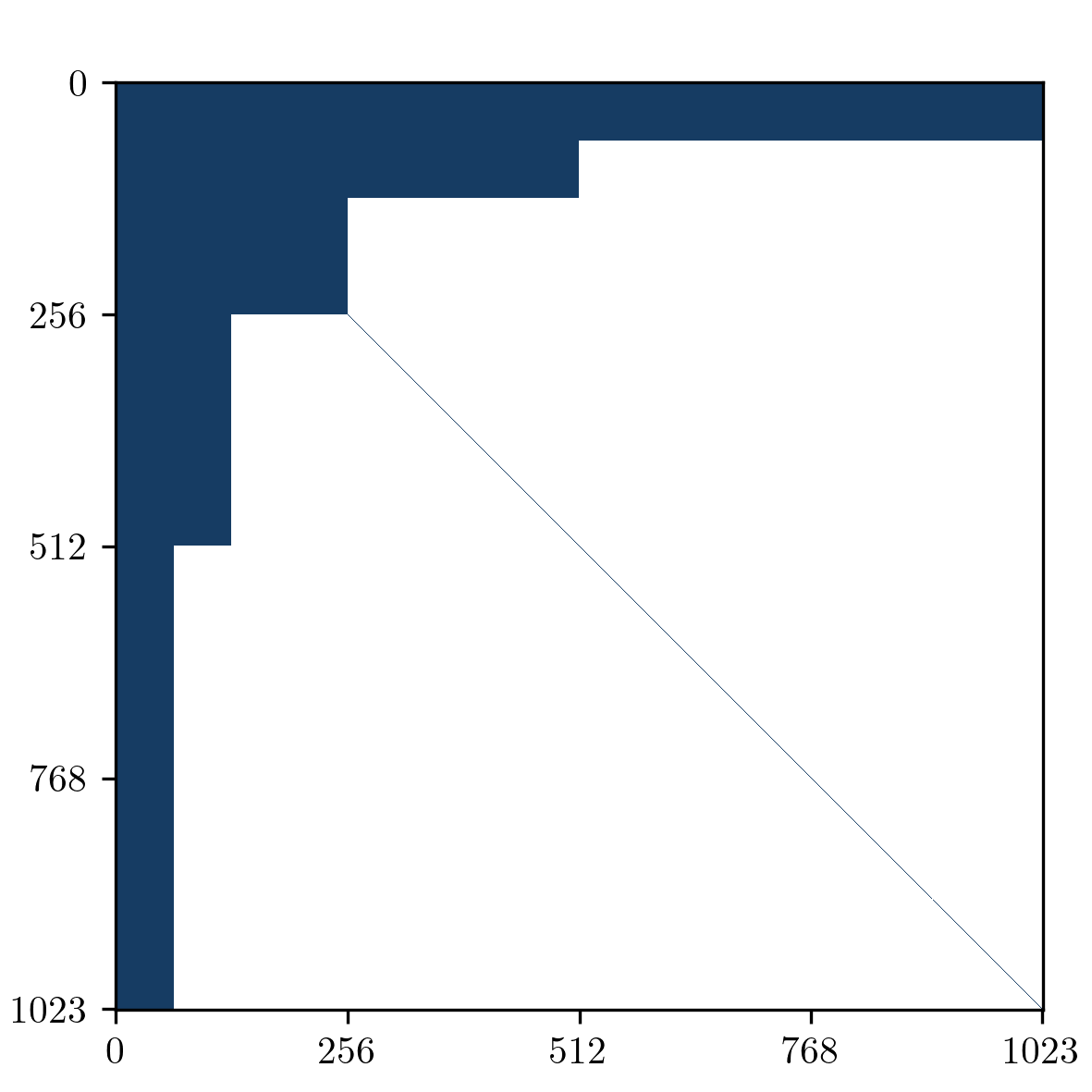}
        \caption{\(18.9\%\) non-zeros at \(L=14\)}
        \label{fig:sparsity-L14}
    \end{subfigure}
    \hfill
    \begin{subfigure}[t]{0.32\textwidth}
        \centering
        \includegraphics[width=\textwidth]{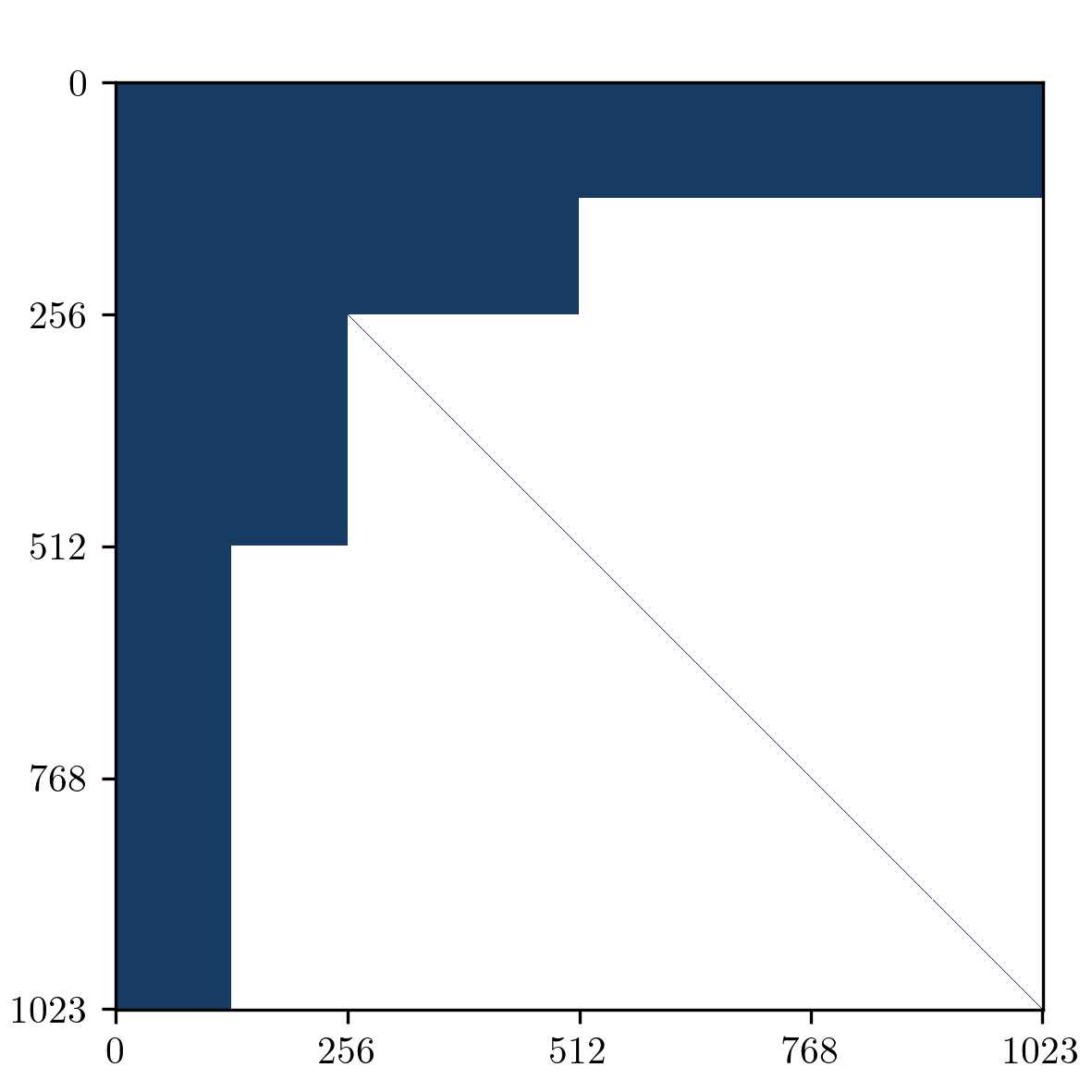}
        \caption{\(31.3\%\) non-zeros at \(L=15\)}
        \label{fig:sparsity-L15}
    \end{subfigure}

    \caption{Sparsity patterns for \(J=9\) and different truncation levels \(L\).}
    \label{fig:sparsity-patterns}
\end{figure}
The second step is to show that this single-layer estimate survives the nonlinear recursion of a DGP. We consider convolution Mat\'ern layers in which the preceding field determines the squared local correlation length through a positive link function \(F\), i.e., the previous sample is composed with $F$ to specify the correlation length on the domain. Under certain assumptions on this link function $F$,  we can control the error propagation through the layers with high probability on  intermediate samples.

The resulting theorem is a stability estimate for the hierarchy as a whole: at every fixed spatial resolution and for every cutoff \(L\), with arbitrarily high prescribed probability the layerwise \(L^2(\Omega)\) sampling errors through the penultimate layer, the terminal covariance error and the terminal sampling error are simultaneously of order  $\mathcal{O}(2^{-rL})$; see Theorem~\ref{thm:dgp-total-error-layer-D}.

The projected covariance \((\bC^L)_+\) is generally dense, but it need not be formed, since Krylov methods such as symmetric Lanczos are capable of computing \(\bigl((\bC^L)_+\bigr)^{1/2}z=h(\bC^L)z\) without full access to \((\bC^L)_+\)~\cite{Higham2008Functions}.

\medskip

The rest of the manuscript is organised as follows:
Section~\ref{sec:dgp_intro} introduces the convolution Matérn hierarchy and the covariance properties required in the analysis. Section~\ref{sec:wavelet-basis} develops the wavelet covariance compression, including the weighted trace estimate, scale truncation, sparsity and positive-semidefinite projection. 
The main result is stated in Section~\ref{sec:main_error_analysis}. Section~\ref{sec:kernel-positivity} collects the kernel properties used in its proof, which is given in Section~\ref{sec:main_proof}.
Section~\ref{sec:numerics} describes the numerical implementation.

\section{Model}\label{sec:dgp_intro}

Throughout, \(\Omega\subset\R^n\) denotes a fixed bounded Lipschitz domain.
All Sobolev spaces are taken over \(\Omega\), or over
\(\Omega\times\Omega\) on product domains. Whenever pointwise continuity of
a kernel is used, the kernel is understood to admit a continuous extension to
\(\overline\Omega\times\overline\Omega\); the associated integral operators
are still defined over \(\Omega\).


\paragraph{Gaussian processes and DGPs.}

For a continuous symmetric positive-semidefinite kernel
\(k:\overline\Omega\times\overline\Omega\to\mathbb R\), we write
\[
    f\sim\operatorname{GP}(0,k)
\]
for a centred Gaussian random field with covariance
\(\mathbb E[f(x)f(y)]=k(x,y)\). Since the mean of the Gaussian process is zero, the covariance function \(k\) completely defines the process. The associated covariance operator is
\[
    (\Cop v)(x):=\int_\Omega k(x,y)v(y)\,dy,
\]
for \(v\in L^2(\Omega)\). We use the terms Gaussian process and Gaussian random field interchangeably on \(\Omega\).

For a fixed depth \(D\in\mathbb N\) with \(D\ge1\), a DGP is a sequence of
random fields \(\{f_\ell\}_{\ell=0}^D\) with Gaussian transition laws
\cite{damianou2013deep,dunlop2018deep}:
\[
    f_0\sim\operatorname{GP}(0,k_0),
    \qquad
    f_\ell\mid f_{\ell-1}
    \sim\operatorname{GP}\bigl(0,k_\ell(\cdot,\cdot;f_{\ell-1})\bigr),
    \quad \ell=1,\ldots,D.
\]
The transition covariance at layer \(\ell\) depends only on the preceding field. We call \(f_D\) the terminal layer.
At layer \(\ell\), \(\Cop_\ell(g)\) denotes the covariance operator associated
with \(k_\ell(\cdot,\cdot;g)\), and \(\Cop_0\) denotes the initial covariance
operator.

\paragraph{Convolution Matérn kernels.}
The Mat\'ern family is a standard class of stationary covariance kernels for
Gaussian random fields; see, for example, \cite{Rasmussen06gaussianprocesses}.
For \(\nu>0\),
\begin{equation}\label{eq:matern defn}
    k_{\M(\nu)}(x,y)
    = \sigma^2 \frac{2^{1-\nu}}{\Gamma(\nu)}
      \left(\sqrt{2\nu}\frac{\|x-y\|_2}{\lambda_0}\right)^{\nu}
      K_{\nu}\left(\sqrt{2\nu}\frac{\|x-y\|_2}{\lambda_0}\right),
\end{equation}
where \(\Gamma\) is the gamma function, \(K_{\nu}\) is the modified Bessel
function of the second kind, and \(\sigma^2,\lambda_0,\nu>0\). Here \(\sigma^2\) is the marginal variance, \(\lambda_0\) the correlation length (the Mat\'ern length-scale parameter), and \(\nu\) the smoothness parameter. At \(x=y\) the kernel is understood as its limiting value \(\sigma^2\). We use the radial profile
notation
\begin{align}\label{eq:maternradial}
    \phi_{\nu,\lambda_0}(r)
    :=
    \sigma^2 \frac{2^{1-\nu}}{\Gamma(\nu)}
    \left(\sqrt{2\nu}\frac{r}{\lambda_0}\right)^\nu
    K_\nu\left(\sqrt{2\nu}\frac{r}{\lambda_0}\right),
\end{align}
so that
\[
    k_{\M(\nu)}(x,y)
    =
    \phi_{\nu,\lambda_0}(\|x-y\|_2).
\]
The profile \(\phi_{\nu,\lambda_0}\) carries the marginal variance, so that \(\phi_{\nu,\lambda_0}(0)=\sigma^2\). The parameterisation of its Gaussian mixture is given in Appendix~\ref{prop:matern-gaussian-mixture}.

Non-stationarity can be introduced by allowing the smoothness, variance or correlation length to vary spatially. Here we vary the correlation length. For \(\nu>0\) and a positive function
\(\lambda:\overline\Omega\to(0,\infty)\), define the convolution Mat\'ern
kernel
\begin{equation}\label{eq:convolution-kernel}
    k_{\conv}^{\lambda,k_{\M(\nu)}}(x,y)
    =
    \frac{
        2^{n/2}\lambda(x)^{n/4}\lambda(y)^{n/4}
    }{
        (\lambda(x)+\lambda(y))^{n/2}
    }\,
    \phi_{\nu,1}\!\left(
        \frac{\|x-y\|_2}
        {\sqrt{(\lambda(x)+\lambda(y))/2}}
    \right).
\end{equation}
Here \(\lambda(x)\) is the squared local correlation-length parameter, so the local correlation length is \(\sqrt{\lambda(x)}\), and
\(\phi_{\nu,1}\) is the Mat\'ern radial profile with unit correlation length.
The prefactor equals \(1\) on the diagonal, so
\(k_{\conv}^{\lambda,k_{\M(\nu)}}(x,x)=\sigma^2\) for every \(x\in\Omega\)
and the marginal variance is \(\sigma^2\) independently of \(\lambda\). If
\(\lambda(x)\equiv\lambda_0^2\), then \eqref{eq:convolution-kernel} reduces to
the stationary Mat\'ern kernel with correlation length \(\lambda_0\). This
is a special case of the more general convolution-kernel construction of
\cite{Paciorek2003,Paciorek2003phd}.

The DGP studied below uses a stationary Mat\'ern initial layer,
\[
    k_0=k_{\M(\nu_0)},
\]
and, for \(\ell=1,\ldots,D\), convolution Mat\'ern layers
\[
    k_\ell(\cdot,\cdot;f_{\ell-1})
    =
    k_{\conv}^{F(f_{\ell-1}),\,k_{\M(\nu_\ell)}}.
\]
Here \(F:\mathbb R\to(0,\infty)\) is a fixed positive link, applied
pointwise according to
\[
    F(g)(x):=F(g(x)).
\]
Thus the realised field at layer \(\ell-1\) determines the local correlation length
at layer \(\ell\). For example, if \(F(t)=e^t\), then the local squared
correlation length is \(\lambda(x)=e^{f_{\ell-1}(x)}\). Layer-dependent links
\(F_\ell\) can be treated by carrying the corresponding constants.

The marginal variance \(\sigma^2\) is fixed across layers. The smoothness
parameters \(\nu_\ell\) may vary with \(\ell\) but are fixed in advance, and
all remaining hyperparameters are deterministic and fixed. Hence, at each
non-initial layer, the covariance depends on the realised sample only through
\(F(f_{\ell-1})\).

\begin{assumption}
\label{as:dgp-fixed-regularity}
Fix a single integer \(r\in\mathbb N\) such that
\begin{equation}\label{eq:fixed-r-assumption}
    \frac n2<r<\nu_\ell,
    \qquad \ell=0,\dots,D.
\end{equation}
We further assume that the positive link satisfies
\begin{equation}\label{eq:dgp-log-link-assumption}
    F\in C^r(\R;(0,\infty)),
    \qquad
    L_F
    :=
    \sup_{t\in\R}\left|\frac{F'(t)}{F(t)}\right|
    <\infty.
\end{equation}
\end{assumption}

The condition \(r>n/2\) provides the Sobolev embedding
\(H^r(\Omega)\hookrightarrow C(\overline\Omega)\) \cite{RAAdams_JJFFournier_2003} and the product estimates of Lemma~\ref{lem:integer-sobolev-product} used for finite wavelet reconstructions and compositions, while
\(r<\nu_\ell\) ensures that the Mat\'ern mixture moments \eqref{eq:matern-mixture-moments} required up to order \(r\) are finite. The bound on \(F'/F\) is equivalent to global Lipschitz continuity of \(\log F\); together with the log-scale kernel estimate of Lemma~\ref{lem:conv-matern-log-scale-lipschitz} it yields global stability of the covariance map with respect to its input field (Lemma~\ref{lem:dgp-L2-covariance-lipschitz}). On bounded \(H^r\)-sets, Sobolev embedding and positivity of \(F\) give uniform upper and lower bounds for \(F(g)\); see Lemma~\ref{lem:dgp-scale-field-bounds}.

Since \(r\) is a positive integer, \eqref{eq:fixed-r-assumption} implies in
particular that \(\nu_\ell>r\ge1\) for every \(\ell\), so
Lemma~\ref{lem:conv-matern-log-scale-lipschitz}, which requires \(\nu>1\),
applies at every layer. Condition \eqref{eq:dgp-log-link-assumption} is
satisfied, for example, by \(F(t)=e^{\alpha t+\beta}\), by \(F(t)=t^2+c\) with
\(c>0\), for which \(L_F=c^{-1/2}\), and by positive shifted and scaled
hyperbolic tangent links \(F(t)=a+b\tanh(ct)\) with \(a>|b|\).

  Assumption~\ref{as:dgp-fixed-regularity} is assumed for the rest of the paper, and \(r\) always denotes the integer fixed there.

\section{Matrix compression}\label{sec:matrix_compression}

\subsection{Scale truncation}\label{sec:det-truncation}

We first study compression for matrices whose coordinates are grouped by
resolution level. The finite wavelet covariance matrices of
Section~\ref{sec:wavelet-basis} will provide a concrete instance of this
setting.

Fix \(j_0\in\mathbb N_0\), and let \(\nabla_j\) be a finite index set for each \(j\ge j_0\). For \(J\ge j_0\), set
\[
    \Lambda_J
    :=
    \{(j,k):j_0\le j\le J,\ k\in\nabla_j\},
    \qquad
    p(J):=\#\Lambda_J
    =\sum_{j=j_0}^J\#\nabla_j,
\]
and write \(j(\iota)=j\) for \(\iota=(j,k)\). Throughout this section, \(\bC\in\R^{p(J)\times p(J)}\) is a symmetric positive semidefinite matrix indexed by \(\Lambda_J\). Define
\begin{equation}\label{eq:dgp-Dbeta-def}
    D_{t,J}
    :=
    \diag\bigl(2^{tj(\iota)}:\iota\in\Lambda_J\bigr),
    \qquad t\in\R,
\end{equation}
abbreviated to \(D_t\) at fixed resolution, so that \(D_{0,J}=I\). We use the exponent \(r\) fixed in Assumption~\ref{as:dgp-fixed-regularity}.

The truncation estimate uses only symmetry, positive semidefiniteness and
these scale weights. The convolution Mat\'ern structure enters later,
through the weighted trace bound of
Lemma~\ref{lem:conv-matern-sobolev-trace}. We first analyse the truncation error, then introduce the projected covariance used for sampling, and finally count the entries the truncation retains.

For \(\iota=(j,k)\), \(\iota'=(j',k')\) and \(L\in\N\), set
\begin{equation}\label{def:scale-compression}
    [\bC^{L}]_{\iota,\iota'}
    :=
    \begin{cases}
        [\bC]_{\iota,\iota'}, & \iota=\iota' \text{ or } j+j'\le L, \\
        0, & \iota\ne\iota' \text{ and } j+j'>L.
    \end{cases}
\end{equation}
The truncation is symmetric, since it depends only on the two levels, and it
retains every diagonal entry, so the marginal variances of the wavelet
coefficients are unchanged by the truncation.

The truncated matrix need not be positive semidefinite.  
We therefore restore positive semidefiniteness of the truncated \(\bC^L\) matrix by clipping its negative eigenvalues to zero as described in Section~\ref{sec:psd-projected-covariance}.

Positive semidefiniteness bounds the Frobenius norm of the weighted matrix \(D_r\bC D_r\) by its trace, independently of the number of wavelets at each scale. On the discarded set, the scale weights then provide the decay in \(L\) appearing below.
\begin{theorem}\label{thm:trace-controlled-scale-tail}
Let \(\bC\) be a symmetric positive semidefinite matrix indexed by
\(\Lambda_J\), let \(\bC^L\) be its diagonal-preserving scale truncation
\eqref{def:scale-compression}, and let \(0\le\beta\le r\). Then
\begin{equation}\label{eq:trace-controlled-scale-tail}
    \|D_\beta(\bC-\bC^L)D_\beta\|_{\op}
    \le
    \|D_\beta(\bC-\bC^L)D_\beta\|_F
    \le
    \tr(D_r\bC D_r)\,2^{-(r-\beta)L}.
\end{equation}
In particular,
\[
    \|\bC-\bC^L\|_{\op}
    \le
    \|\bC-\bC^L\|_F
    \le
    \tr(D_r\bC D_r)\,2^{-rL}.
\]
\end{theorem}
\begin{proof}
Set \(Q:=D_r\bC D_r\), which is positive semidefinite because \(\bC\) is and
\(D_r\) is diagonal with positive entries.  The matrix \(\bC-\bC^L\) vanishes
off the discarded set, on which
\[
    [D_\beta(\bC-\bC^L)D_\beta]_{\iota,\iota'}
    = 2^{\beta(j(\iota)+j(\iota'))}[\bC]_{\iota,\iota'}
    = 2^{-(r-\beta)(j(\iota)+j(\iota'))}Q_{\iota,\iota'} .
\]
Since \(j(\iota)+j(\iota')>L\) there and \(\beta\le r\),
\[
    \|D_\beta(\bC-\bC^L)D_\beta\|_F^2
    = \sum_{\substack{\iota\ne\iota'\\ j(\iota)+j(\iota')>L}}
      2^{-2(r-\beta)(j(\iota)+j(\iota'))}|Q_{\iota,\iota'}|^2
    \le 2^{-2(r-\beta)L}\|Q\|_F^2.
\]
The result follows by \(\|A\|_{F}\leq \tr(A)\) for positive semi definite symmetric matrices and taking square
roots gives the Frobenius estimate in
\eqref{eq:trace-controlled-scale-tail}. The operator-norm bound follows
from \(\|\cdot\|_{\op}\le\|\cdot\|_F\).
\end{proof}

To estimate the computational cost of working with the truncated matrix, we assume
\begin{equation}\label{eq:graded-level-cardinality}
    \#\nabla_j\simeq2^{nj},
    \qquad j\ge j_0,
\end{equation}
with constants independent of \(j\). Write \(\nnz(A)\) for the number of nonzero entries of \(A\). We bound the number of entries retained by the truncation, which is an upper bound for \(\nnz(\bC^L)\).

\begin{theorem}\label{thm:scale-only-sparsity}
Let \(\bC^L\) be the diagonal-preserving scale truncation
\eqref{def:scale-compression}.  If \(L\ge2J\) then nothing is truncated and
\(\bC^L=\bC\).  If \(L<2J\), then the number of nonzero entries satisfies
\[
    \nnz(\bC^L)
    \lesssim
    p(J)
    +
    \begin{cases}
        (L+1)2^{nL}, & L\le J,\\[1mm]
        (2J-L+1)2^{nL}, & J\le L<2J,
    \end{cases}
\]
with implicit constants depending only on \(n\), \(j_0\) and the constants
in \eqref{eq:graded-level-cardinality}.
\end{theorem}

\begin{proof}
If \(L\ge2J\) then \(j+j'\le2J\le L\) for every pair, so no entry is
discarded.  Assume \(L<2J\).  The truncation retains all \(p(J)\) diagonal
entries, and the retained off-diagonal entries lie in the set of pairs with
\(j+j'\le L\).  Hence, using \(\#\nabla_j\simeq2^{nj}\) for each admissible
\(j\in\{j_0,\ldots,J\}\),
\begin{equation}\label{eq:scale_sparsity}
    \nnz(\bC^{L})
    \le
    p(J)+
    \sum_{\substack{j,j'=j_0\\ j+j'\le L}}^{J}
    \#\nabla_j\#\nabla_{j'}
    \lesssim
    p(J)+
    \sum_{\substack{j,j'=j_0\\ j+j'\le L}}^{J}
    2^{nj}2^{nj'}.
\end{equation}
If \(L\le J\), then \(j+j'\le L\) implies \(j,j'\le L\), and
\[
    \nnz(\bC^{L})
    \lesssim
    p(J)+\sum_{j=j_0}^{L}2^{nj}\sum_{j'=j_0}^{L-j}2^{nj'}
    \lesssim
    p(J)+\sum_{j=j_0}^{L}2^{nj}2^{n(L-j)}
    \le
    p(J)+(L+1)2^{nL}.
\]
If \(J\le L<2J\), then for \(j\le L-J\) the constraint \(j'\le L-j\) is
inactive, so \(j'\le J\), while for \(j>L-J\) one has \(L-j<J\) and the
active constraint is \(j'\le L-j\).  Splitting \eqref{eq:scale_sparsity} at
\(j=L-J\) gives
\[
\begin{aligned}
    \nnz(\bC^{L})
    &\lesssim
    p(J)+
    \sum_{j=j_0}^{L-J}2^{nj}\sum_{j'=j_0}^{J}2^{nj'}
    +
    \sum_{j=L-J+1}^{J}2^{nj}\sum_{j'=j_0}^{L-j}2^{nj'}\\
    &\lesssim
    p(J)+\sum_{j=j_0}^{L-J}2^{nj}2^{nJ}
    +\sum_{j=L-J+1}^{J}2^{nj}2^{n(L-j)}
    \lesssim
    p(J)+(2J-L+1)2^{nL}.
\end{aligned}
\]
\end{proof}
\begin{corollary}\label{cor:scale-only-sparsity-regimes}
Let \(c\in\mathbb N_0\) be fixed and set \(L=J+c\).
Then, for all sufficiently large \(J\),
\[
    \nnz(\bC^{J+c})
    \lesssim
    2^{nc}\,p(J)\log p(J),
\]
so that \(L=J+O(1)\) gives \(\nnz(\bC^L)=O(p(J)\log p(J))\).
\end{corollary}

\begin{proof}
Since \(p(J)\simeq2^{nJ}\), we have \(J\simeq\log p(J)\).  Taking
\(L=J+c<2J\) in the second case of
Theorem~\ref{thm:scale-only-sparsity} gives \(2J-L+1=J-c+1\lesssim\log p(J)\)
and \(2^{nL}=2^{nc}2^{nJ}\simeq2^{nc}p(J)\).
\end{proof}

\subsection{Restoring positive semidefiniteness}
\label{sec:psd-projected-covariance}

The truncation described in the previous subsection can result in indefinite matrices, even if $\bC$ is semidefinite. In the following, it will be useful to restore positive definiteness of the matrices through projection. To that end, let \(\mathbb S^p\) denote the real symmetric \(p\times p\) matrices and
\(\mathbb S^p_+\subset\mathbb S^p\) the positive-semidefinite cone.  For
\(A\in\mathbb S^p\) with spectral expansion \(A=\sum_k\lambda_kP_k\), where
the \(P_k\) are the orthogonal projections onto the eigenspaces, define the
positive and negative spectral parts by
\begin{equation}\label{eq:positive-spectral-part}
    A_\pm:=\sum_k\lambda_k^{\pm}P_k,
    \qquad
    \lambda^+:=\max\{\lambda,0\},
    \quad
    \lambda^-:=\max\{-\lambda,0\}.
\end{equation}
Then \(A=A_+-A_-\) and \(A_+A_-=0\).  For ease of notation later we write \(\widehat{\bC}^{L}:=(\bC^L)_+\).

\begin{lemma}\label{lem:psd-projection-nonexpansive}
The map \(\Pi_+:\mathbb S^p\to\mathbb S^p_+\), \(\Pi_+(A):=A_+\), is the
metric projection onto \(\mathbb S^p_+\) in the Frobenius inner product, and
is nonexpansive:
\begin{equation}\label{eq:psd-projection-nonexpansive}
    \|\Pi_+(A)-\Pi_+(B)\|_F
    \le
    \|A-B\|_F,
    \qquad A,B\in\mathbb S^p .
\end{equation}
Consequently, if \(\bC\succeq0\), then
\begin{equation}\label{eq:projected-scale-tail}
    \|\bC-\widehat{\bC}^{L}\|_F
    \le
    \|\bC-\bC^L\|_F
    \le
    \tr(D_r\bC D_r)2^{-rL}.
\end{equation}
\end{lemma}
\begin{proof}
By \cite[Theorem~2.1]{Higham1988NearestPSD}, the positive spectral part \(A_+\) is the metric projection of a symmetric matrix \(A\) onto \(\mathbb S_+^p\) in the Frobenius inner product. Since the cone is closed and convex, its metric projection satisfies \eqref{eq:psd-projection-nonexpansive}. Since \(\Pi_+(\bC)=\bC\) for \(\bC\succeq0\), applying this estimate to \(\bC\) and \(\bC^L\), followed by Theorem~\ref{thm:trace-controlled-scale-tail}, proves \eqref{eq:projected-scale-tail}.
\end{proof}
\begin{remark}
Although the truncation preserves the diagonal, the projected covariance
need not. Since
\(\widehat{\bC}^L-\bC^L=(\bC^L)_-\succeq0\), the diagonal can only
increase. Moreover,
\[
    [(\bC^L)_-]_{\iota,\iota}
    \le
    \|(\bC^L)_-\|_{\op}
    \le
    \|\bC-\bC^L\|_{\op},
\]
where the second inequality follows from Weyl's inequality and
\(\bC\succeq0\). Hence, by Theorem~\ref{thm:trace-controlled-scale-tail},
\[
    0
    \le
    [\widehat{\bC}^L]_{\iota,\iota}-[\bC]_{\iota,\iota}
    \le
    \tr(D_r\bC D_r)2^{-rL}.
\]
Thus the marginal variances are perturbed by at most the same order as
the truncation error.
\end{remark}

\section{Wavelet discretisation}\label{sec:wavelet-basis}
We introduce the discretisation in wavelet coordinates and the sampling procedure.
\subsection{Wavelets}\label{sec:wavelet-discretisation}

With the same coarsest level \(j_0\) as above, we consider a biorthogonal multiresolution basis on \(L^2(\Omega)\)
\cite{CohenDaubechiesFeauveau1992,Stevenson2007Composite}: This means we use sequences of primal and dual basis functions
\begin{align*}
\Psi = (\Psi_j)_{j\geq j_0},\,\Psi_j=\{\psi_{j,k}\}_{k\in\nabla_j}\quad\text{and}\quad \widetilde\Psi = (\widetilde \Psi_j)_{j\geq j_0},\,
\widetilde\Psi_j=\{\widetilde\psi_{j,k}\}_{k\in\nabla_j},
\end{align*}
with the crucial biorthogonality property
\begin{align}
    \langle\psi_{j,k},\widetilde\psi_{j',k'}\rangle_{L^2(\Omega)}
    &=\delta_{j,j'}\delta_{k,k'},
    \label{eq:wavelet-biorthogonality}
\end{align}
such that for each level 
\begin{align}
    \#\nabla_j
    &\simeq2^{nj},
    \qquad j\ge j_0.
    \label{eq:dim-nablaj}
\end{align}
Naturally, these basis elements define a corresponding sequence of subspaces $W_j :={\rm span}\{\psi_{j,k}: k\in\nabla_j\}$ and $\widetilde W_j :={\rm span}\{\widetilde \psi_{j,k}: k\in\nabla_j\}$, which form the corresponding multiresolution analysis $V_j:= \oplus_{i=j_0}^j W_i$, $\widetilde V_j:= \oplus_{i=j_0}^j \widetilde W_i$ satisfying
\[
    V_{j_0}\subset V_{j_0+1}\subset\cdots\subset L^2(\Omega),
    \qquad
    \widetilde V_{j_0}\subset\widetilde V_{j_0+1}
    \subset\cdots\subset L^2(\Omega).
\]
Clearly, there holds 
\(\dim V_j=\dim\widetilde V_j\simeq2^{nj}\) and the level \(j\) corresponds to a spatial scale \(2^{-j}\).

We use the index sets \(\Lambda_J\), the level map \(j(\iota)\)
and the matrix size \(p(J)\) introduced in
Section~\ref{sec:matrix_compression}, with \(\nabla_j\) now
indexing the wavelets at level \(j\). The full wavelet index set is
\[
    \mathcal J:=\{(j,k):j\ge j_0,\ k\in\nabla_j\}.
\]
The relation \eqref{eq:dim-nablaj} gives
\eqref{eq:graded-level-cardinality}, and \(p(J)=\dim V_J\simeq2^{nJ}\).
For \(f\in L^2(\Omega)\), write
\[
    a_J(f)
    :=
    \bigl(\langle f,\psi_\iota\rangle_{L^2(\Omega)}\bigr)
    _{\iota\in\Lambda_J}
\]
for its primal-coefficient vector. The diagonal scale matrices
\(D_{t,J}\) are defined in \eqref{eq:dgp-Dbeta-def}.
\begin{assumption}
\label{as:wavelet-stability}
There are constants \(1\leq A_0,A_r,S_{\rm Riesz}<\infty\), independent of \(J\), such that
\begin{align}
    \|a_J(v)\|_2^2
    &\le
    A_0\|v\|_{L^2(\Omega)}^2,
    \qquad
    v\in L^2(\Omega),
    \label{eq:wavelet-analysis-L2}
    \\
    \|D_{r,J}a_J(v)\|_2^2
    &\le
    A_r\|v\|_{H^r(\Omega)}^2,
    \qquad
    v\in H^r(\Omega),
    \label{eq:wavelet-analysis-Hr}
\end{align}
and such that, for every coefficient vector
\(a=(a_\iota)_{\iota\in\Lambda_J}\in\R^{\Lambda_J}\),
\begin{align}
    \Bigl\|
        \sum_{\iota\in\Lambda_J}
        a_\iota\widetilde\psi_\iota
    \Bigr\|_{L^2(\Omega)}^2
    &\le
    S_{\rm Riesz}\|a\|_2^2,
    \label{eq:wavelet-synthesis-L2}
    \\
    \Bigl\|
        \sum_{\iota\in\Lambda_J}
        a_\iota\widetilde\psi_\iota
    \Bigr\|_{H^r(\Omega)}^2
    &\le
    S_{\rm Riesz}\|D_{r,J}a\|_2^2.
    \label{eq:wavelet-synthesis-Hr}
\end{align}
Here \(r\) is the integer fixed in
Assumption~\ref{as:dgp-fixed-regularity}.
\end{assumption}

The usual two-sided Sobolev Riesz-basis bounds, see e.g. {\cite[eq.~(3.11)]{Harbrecht2024}}, at orders \(0\) and \(r\)
imply these estimates. Assumptions~\ref{as:dgp-fixed-regularity} and
\ref{as:wavelet-stability} constrain \(r\) jointly: it must satisfy
\(n/2<r<\min_{\ell=0,\ldots,D}\nu_\ell\) and lie within the Sobolev
stability range of the wavelet pair.
We impose Assumption~\ref{as:wavelet-stability} throughout the remainder of the paper.

By \eqref{eq:wavelet-synthesis-Hr}, every finite dual-wavelet reconstruction belongs to \(H^r(\Omega)\). Since \(r>n/2\), it has a continuous representative on \(\overline\Omega\).

 If
\(
    f_J=\sum_{\iota\in\Lambda_J}a_\iota\widetilde\psi_\iota,
\)
then \eqref{eq:wavelet-biorthogonality} gives \(a=a_J(f_J)\). If \(f\) is a
centred Gaussian field with covariance operator \(\Cop\), then \(a_J(f)\) is a
centred Gaussian vector, with covariance matrix
\begin{equation}\label{eq:C_J-def}
\begin{aligned}
    \relax[\bC_J]_{\iota,\iota'}
    &:=\E[(a_J(f))_\iota(a_J(f))_{\iota'}]=\langle\Cop\psi_\iota,\psi_{\iota'}\rangle_{L^2(\Omega)}\\
    &=\iint_{\Omega\times\Omega}
    k(x,y)\psi_\iota(x)\psi_{\iota'}(y)\,dx\,dy,
    \quad \iota,\iota'\in\Lambda_J.
\end{aligned}
\end{equation}
where \(k\) is the covariance kernel of \(\Cop\).

\paragraph{Weighted trace and sample regularity.}
Let \(f_J:=\sum_{\iota\in\Lambda_J}(a_J(f))_\iota\widetilde\psi_\iota\) be the finite reconstruction of \(f\). By biorthogonality, \(a_J(f_J)=a_J(f)\), and the covariance identity \eqref{eq:C_J-def} gives
\[
    \E\|D_{r,J}a_J(f)\|_2^2
    =\tr(D_{r,J}\bC_JD_{r,J}).
\]
Applying the analysis bound \eqref{eq:wavelet-analysis-Hr} to \(f_J\) and the synthesis bound \eqref{eq:wavelet-synthesis-Hr} to its coefficients therefore gives
\begin{equation}\label{eq:trace-sobolev-interpretation}
    A_r^{-1}\tr(D_{r,J}\bC_JD_{r,J})
    \le
    \E\|f_J\|_{H^r(\Omega)}^2
    \le
    {S_{\rm Riesz}}\tr(D_{r,J}\bC_JD_{r,J}).
\end{equation}
Thus the weighted trace is equivalent to the mean-square \(H^r\)-norm of the finite reconstruction, with constants independent of \(J\). Note, this interpretation requires no assumption that the original field \(f\) has \(H^r\)-regular samples, since the analysis and synthesis bounds are applied only to its finite reconstruction.

\begin{remark}
Using entrywise decay gives a weaker version of Theorem~\ref{thm:trace-controlled-scale-tail}. By duality and \eqref{eq:wavelet-analysis-Hr}, \(\|\psi_\iota\|_{(H^r(\Omega))'}\le A_r^{1/2}2^{-rj(\iota)}\), so \(|[\bC]_{\iota,\iota'}|\le A_r\|\Cop\|_{(H^r)'\to H^r}\, 2^{-r(j(\iota)+j(\iota'))}\).
Summing over the discarded entries with \eqref{eq:graded-level-cardinality} gives, for \(0\le\beta<r-n/2\),
\[
    \|D_\beta(\bC-\bC^L)D_\beta\|_F
    \lesssim
    \|\Cop\|_{(H^r)'\to H^r}\,(L+1)^{1/2}2^{-(r-\beta-n/2)L},
\]
compared with \(\mathcal{O}(2^{-(r-\beta)L})\) from the trace argument.
\end{remark}

\subsection{Sampling}

Section~\ref{sec:matrix_compression} gives deterministic approximation estimates for a single truncated covariance matrix. We now apply this construction to the finite-dimensional wavelet DGP. Since compression changes the field supplied to the covariance in the next layer, the exact and approximate DGP constructions must be compared recursively.

We discretise the convolution Mat\'ern DGP of Section~\ref{sec:dgp_intro}, given by
\begin{align*}
    f_0&\sim\mathcal{GP}(0,k_{\M(\nu_0)}),\\
    f_\ell\mid f_{\ell-1}
    &\sim\mathcal{GP}\!\left(
        0,k_{\conv}^{F(f_{\ell-1}),k_{\M(\nu_\ell)}}
    \right),
    \qquad \ell=1,\dots,D,
\end{align*}
in the wavelet coordinates introduced in Section~\ref{sec:wavelet-discretisation}.

For every resolution level \(J'\ge j_0\), define the initial covariance
matrix
\[
    \bC_{0,J'}
    :=
    \bigl(
        \langle \Cop_0\psi_\iota,\psi_{\iota'}\rangle_{L^2(\Omega)}
    \bigr)_{\iota,\iota'\in\Lambda_{J'}},
\]
and, for \(\ell=1,\dots,D\) and \(g\in C(\overline\Omega)\), define
\[
    \bC_{\ell,J'}(g)
    :=
        \left(\langle
            \Cop_\ell(g)\psi_\iota,\psi_{\iota'}\rangle_{L^2(\Omega)}
    \right)_{\iota,\iota'\in\Lambda_{J'}}.
\]
By \eqref{eq:wavelet-synthesis-Hr}, every finite wavelet reconstruction appearing below lies in \(H^r(\Omega)\hookrightarrow C(\overline\Omega)\), so these matrices are defined for all inputs.
Let \(\mathbf C_{0,J'}^L\) and \(\mathbf C_{\ell,J'}^L(g)\) denote their
scale truncations according to \eqref{def:scale-compression}.  For the
remainder of this section fix the resolution \(J\) and suppress this subscript,
writing
\[
    \mathbf C_0:=\mathbf C_{0,J},\qquad
    \mathbf C_0^L:=\mathbf C_{0,J}^L,\qquad
    \mathbf C_\ell(g):=\mathbf C_{\ell,J}(g),\qquad
    \mathbf C_\ell^L(g):=\mathbf C_{\ell,J}^L(g).
\]
Define the projected sampling covariances
\begin{equation}\label{eq:dgp-projected-covariances}
    \widehat{\mathbf C}_0^L
    :=
    (\mathbf C_0^L)_+,
    \qquad
    \widehat{\mathbf C}_\ell^L(g)
    :=
    (\mathbf C_\ell^L(g))_+,
    \quad \ell=1,\dots,D.
\end{equation}

Since \(F\) is strictly positive and every finite wavelet reconstruction is continuous, the exact matrices are positive definite by Lemma~\ref{lem:conv-matern-integral-strict-positivity} and the linear independence of the primal wavelets \eqref{eq:wavelet-biorthogonality} make the exact matrices positive definite. The truncated matrix \(\mathbf C^L\) may be indefinite, but the projected matrices in \eqref{eq:dgp-projected-covariances} are positive semidefinite for every input and every \(L\), so each defines a, possibly degenerate, Gaussian distribution.

To sample from the DGP, we draw normal random numbers
\[
    z_0,\dots,z_D
    \stackrel{\mathrm{iid}}{\sim}
    \mathcal N(0,I_{p(J)}).
\] 
and define the exact and discretised layerwise sample coefficients by
\begin{align*}
    a_0
    &:=\mathbf C_0^{1/2}z_0,
    &
    \widehat a_0^L
    &:=(\widehat{\mathbf C}_0^L)^{1/2}z_0,
    \\
    a_\ell
    &:=\mathbf C_\ell(f_{\ell-1})^{1/2}z_\ell,
    &
    \widehat a_\ell^L
    &:=\widehat{\mathbf C}_\ell^L(\widehat f_{\ell-1}^L)^{1/2}z_\ell,
    \qquad \ell=1,\dots,D.
\end{align*}
The DGP samples can then be reconstructed by
\begin{equation}\label{eq:dgp-finite-reconstruction}
    f_\ell
    :=
    \sum_{\iota\in\Lambda_J}(a_\ell)_\iota\widetilde\psi_\iota,
    \qquad
    \widehat f_\ell^L
    :=
    \sum_{\iota\in\Lambda_J}(\widehat a_\ell^L)_\iota
    \widetilde\psi_\iota.
\end{equation}
In the following, we will denote the sequence of layerwise components $(f_\ell, a_\ell, \bC_\ell)_\ell$ as \emph{chains} of the DGP.

\section{Main result}\label{sec:main_error_analysis}

Our main result provides, with probability at least \(1-\eta\), a bound on the difference between the untruncated and approximate chains. The constants in the estimate depend on the failure probability \(\eta\) and deteriorate as \(\eta \to 0\).

\begin{theorem}
\label{thm:dgp-total-error-layer-D}
Suppose Assumptions~\ref{as:dgp-fixed-regularity} and
\ref{as:wavelet-stability} hold, let \(J\ge j_0\) and \(\eta\in(0,1)\) and let \(f_\ell\) and \(\widehat f_\ell^L\) be the
chains \eqref{eq:dgp-finite-reconstruction}.
There are finite constants \(P_0,\ldots,P_D\) and \(K^{\mathrm{cov}}_D\), independent of \(L\), such that for every \(L\in\N\), with probability at least \(1-\eta\),
\begin{equation}\label{eq:dgp-final-L2-error}
    \|f_\ell-\widehat f_\ell^L\|_{L^2(\Omega)}
    \le
    P_\ell 2^{-rL},
    \qquad
    \ell=0,\dots,D,
\end{equation}
\begin{equation}\label{eq:dgp-final-covariance-error}
    \|\CellJ{D}(f_{D-1})
    -\widehat{\mathbf C}_D^L(\widehat f_{D-1}^L)\|_F
    \le
    K_D^{\mathrm{cov}}2^{-rL}.
\end{equation} 
\end{theorem}
The constants in this theorem depend on \(J\) only through the lower spectral bounds of Lemma~\ref{lem:dgp-finite-resolution-coercivity}; explicit expressions are given in \eqref{eq:dgp-propagation-prefactor}--\eqref{eq:dgp-terminal-cov-prefactor}.
\paragraph{Outline of the proof.}
Write \(e_\ell:=\|f_\ell-\widehat f_\ell^L\|_{L^2(\Omega)}\) for \(\ell\ge0\). Nonexpansiveness of \(\Pi_+\) gives
\begin{equation}\label{eq:outline-splitting}
    \|\bC_\ell(f_{\ell-1})-\widehat{\bC}_\ell^L(\widehat f_{\ell-1}^L)\|_F
    \le
    \|\bC_\ell(f_{\ell-1})-\bC_\ell^L(f_{\ell-1})\|_F+\|\bC_\ell^L(f_{\ell-1})-\bC_\ell^L(\widehat f_{\ell-1}^L)\|_F.
\end{equation}
A weighted trace estimate bounds the first term by \(\mathcal{O}(2^{-rL})\) whenever \(f_{\ell-1}\) lies in a bounded \(H^r\)-set (Lemma~\ref{lem:dgp-uniform-trace-balls}), and Gaussian concentration shows that the exact chain lies in such sets with high probability (Proposition~\ref{prop:dgp-good-event}). A global Lipschitz estimate bounds the second term by a multiple of \(e_{\ell-1}\) (Lemma~\ref{lem:dgp-L2-covariance-lipschitz}). A lower eigenvalue bound for \(\bC_\ell(f_{\ell-1})\) (Lemma~\ref{lem:dgp-finite-resolution-coercivity}) and the square-root estimate~\eqref{eq:frobenius-square-root-one-gap} below turn \eqref{eq:outline-splitting} into the recursion \(e_\ell\le\alpha_\ell2^{-rL}+\gamma_\ell e_{\ell-1}\) (Proposition~\ref{prop:dgp-one-step-minimal}), and induction over the layers proves the theorem.

\section{Kernel properties}\label{sec:kernel-positivity}
\begingroup
\emergencystretch=2em
By Theorem~\ref{thm:trace-controlled-scale-tail} and \eqref{eq:trace-sobolev-interpretation}, the weighted trace \(\tr(D_{r,J}\bC_JD_{r,J})\) controls both the compression error and the expected squared \(H^r\)-norms of the sampled fields. The latter control yields the high probability sample-norm bounds needed to propagate errors through successive layers.
Although compression alone requires only \(\|D_{r,J}\bC_JD_{r,J}\|_F\), a bound on this smaller quantity does not in general control the trace uniformly in \(J\).
For example, \(\operatorname{diag}(1,1/2,\ldots,1/p)\) has bounded Frobenius norm but diverging trace as \(p\to\infty\). We therefore seek a trace bound uniform in \(J\) and explicit in \(\lambda\), supplying both the compression estimate and the sample bounds needed in Proposition~\ref{prop:dgp-good-event}.
\par
\endgroup

The following Lemma~\ref{lem:kernel_properties} is the technical heart of the proof of the main result. 
 All five parts of the lemma are proved in
Appendix~\ref{app:kernel-properties-proof}.

For \(B>0\), set
\[
    \mathcal B_r(B)
    :=
    \{g\in H^r(\Omega):\|g\|_{H^r(\Omega)}\le B\}.
\]
Since \(r>n/2\), Sobolev embedding gives
\(
    \|g\|_{L^\infty(\Omega)}
    \le
    C^{\mathrm{emb}}_{\Omega,r}\|g\|_{H^r(\Omega)}.
\)

For \(r<\nu\), define the quantity
\begin{equation*}
    \mathcal M_{r,\nu}
    :=
    \sum_{q=0}^{r}
    \left(\frac{\nu}{2}\right)^q
    \frac{\Gamma(\nu-q)}{\Gamma(\nu)}
\end{equation*}
as well as
\begin{align}\label{eq:dgp-trace-majorant}
    \mathcal T_\ell(B)
    &:=
    1+
    \sup_{g\in\mathcal B_r(B)}
    \sup_{J'\ge j_0}
    \tr\!\left(
        D_{r,J'}\mathbf C_{\ell,J'}(g)D_{r,J'}
    \right),\quad\text{ for }\ell=1,\ldots,D,\\
     \mathcal T_0
    &:=
    1+
    \sup_{J'\ge j_0}
    \tr\!\left(
        D_{r,J'}\mathbf C_{0,J'}D_{r,J'}
    \right)\nonumber.
\end{align}

\begin{lemma}\label{lem:kernel_properties}
Suppose Assumption~\ref{as:dgp-fixed-regularity} and Assumption~4.1 hold.
\begin{enumerate}[label=(\roman*),ref=\thetheorem(\roman*)]

\item\label{lem:conv-matern-sobolev-trace}
Let \(r<\nu\) and \(\lambda\in H^r(\Omega)\) satisfy
\(\lambda(x)\geq\underline\lambda>0\) for every
\(x\in\overline\Omega\), and let \(\bC_J\) be the matrix
\eqref{eq:C_J-def} for \(k_{\conv}^{\lambda,k_{\M(\nu)}}\). Then
\begin{equation}\label{eq:conv-matern-weighted-trace-direct}
  \sup_{J\geq j_0}\tr\bigl(D_{r,J}\bC_JD_{r,J}\bigr)
  \lesssim
  \sigma^2(1+\underline\lambda^{-1})^{2r}
  (1+\|\lambda\|_{H^r(\Omega)})^{2r}\mathcal M_{r,\nu}.
\end{equation}
The implicit constant depends only on \(\Omega,n,r,A_r\).

\item\label{lem:dgp-scale-field-bounds}\label{lem:dgp-uniform-trace-balls}
For every \(B>0\), \(\mathcal T_0<\infty\) and
\(\mathcal T_\ell(B)<\infty\) for \(\ell=1,\ldots,D\).
For every \(g\in\mathcal B_r(B)\), \(J\geq j_0\), and
\(L\in\mathbb N\),
\begin{equation}\label{eq:dgp-uniform-tail-bound}
  \|\CellJ{\ell}(g)-\CellJL{\ell}(g)\|_F
  \leq\mathcal T_\ell(B)2^{-rL},
  \qquad
  \|\mathbf C_0-\mathbf C_0^L\|_F
  \leq\mathcal T_0 2^{-rL}.
\end{equation}
More precisely, there is a finite nondecreasing function
\(R_{F,r}:[0,\infty)\to[0,\infty)\) such that
\begin{equation}\label{eq:dgp-composition-bound}
  \|F(g)\|_{H^r(\Omega)}\leq R_{F,r}(B),
  \qquad g\in\mathcal B_r(B).
\end{equation}
Set
\[
  \underline\lambda(B)
    :=\inf_{|t|\leq C^{\mathrm{emb}}_{\Omega,r}B}F(t),
  \qquad
  \overline\lambda(B)
    :=\sup_{|t|\leq C^{\mathrm{emb}}_{\Omega,r}B}F(t).
\]
For every \(g\in\mathcal B_r(B)\) and \(x\in\overline\Omega\),
\[
  0<\underline\lambda(B)\leq F(g(x))
    \leq\overline\lambda(B)<\infty,
\]
and
\begin{equation}\label{eq:dgp-trace-majorant-explicit}
  \mathcal T_\ell(B)
  \lesssim1+\sigma^2(1+\underline\lambda(B)^{-1})^{2r}
     (1+R_{F,r}(B))^{2r}\mathcal M_{r,\nu_\ell}.
\end{equation}
The implicit constant depends only on \(\Omega,n,r,A_r\).

\item\label{lem:dgp-L2-covariance-lipschitz}
For each \(\ell\in\{1,\ldots,D\}\), there is a constant
\(\Lip_\ell<\infty\), independent of \(J,L\) and the input fields,
such that, for all \(f,g\in C(\overline\Omega)\),
\begin{equation}\label{eq:dgp-global-frob-lipschitz}
\begin{aligned}
  \|\widehat{\mathbf C}_\ell^L(f)
      -\widehat{\mathbf C}_\ell^L(g)\|_F
  &\leq\|\CellJL{\ell}(f)-\CellJL{\ell}(g)\|_F\\&
  \leq\|\CellJ{\ell}(f)-\CellJ{\ell}(g)\|_F
  \leq\Lip_\ell\|f-g\|_{L^2(\Omega)}.
\end{aligned}
\end{equation}
One may take
\begin{equation}\label{eq:dgp-global-Lip-constant}
  \Lip_\ell
  :=2A_0|\Omega|^{1/2}\sigma^2
       \left(\frac n4+\frac1e\right)L_F,
\end{equation}
which is independent of \(\ell\).

\item\label{lem:conv-matern-integral-strict-positivity}
Let \(\nu>0\) and suppose that
\begin{equation}\label{eq:scale-field-bounds}
  \lambda\in C(\overline\Omega),
  \qquad 0<\underline\lambda\leq\lambda(x)
    \leq\overline\lambda<\infty,
  \qquad x\in\overline\Omega.
\end{equation}
Then \(k=k_{\conv}^{\lambda,k_{\M(\nu)}}\) is continuous,
symmetric, positive definite on $L^2(\Omega)$, i.e., there holds \(\int_\Omega\int_\Omega k(x,y)u(x)u(y)dxdy>0\) for every \(0\ne u\in L^2(\Omega)\) (note that positive definiteness on a finite set is already shown in~\cite{Paciorek2003}). In particular, the matrices \(\bC_J\) in
\eqref{eq:C_J-def} are positive definite.

\item\label{lem:dgp-finite-resolution-coercivity}
For every fixed resolution level \(J\):
\begin{enumerate}
\item the initial covariance matrix \(\mathbf C_0\) is symmetric
positive definite, with \(\kappa_{0,J}:=\lambda_{\min}(\mathbf C_0)>0\);
\item for every \(\ell\in\{1,\ldots,D\}\) and \(B>0\),
\[
  \kappa_{\ell,J}(B)
  :=\inf_{\|g\|_{H^r(\Omega)}\leq B}
       \lambda_{\min}(\CellJ{\ell}(g))>0.
\]
\end{enumerate}

\end{enumerate}
\end{lemma}

\section{Error analysis}\label{sec:main_proof}
To apply parts \ref{lem:conv-matern-sobolev-trace} and \ref{lem:conv-matern-integral-strict-positivity} of Lemma~\ref{lem:kernel_properties} at layer \(\ell\), we take \(\lambda=F(f_{\ell-1})\), the squared local correlation length at layer \(\ell\). We therefore need compressibility and positivity uniformly over the bounded subsets of \(H^r(\Omega)\) in which the exact chain lies with high probability (Proposition~\ref{prop:dgp-good-event}). We emphasise that no regularity control of the approximated chain is needed, due to the global Lipschitz stability of the covariance matrices shown in Lemma~\ref{lem:dgp-L2-covariance-lipschitz}.
\subsection{Bounds on sample norms}
We first bound the \(H^r\)-norms of the exact chain of layerwise DGP components with high probability. On the resulting events, the trace and lower eigenvalue bounds of Lemma~\ref{lem:kernel_properties} hold with deterministic constants. For \(\ell=0,\ldots,D\), let \(\mathscr F_\ell:=\sigma(z_0,\ldots,z_\ell)\).
By Lemma~\ref{lem:dgp-L2-covariance-lipschitz},
Lemma~\ref{lem:psd-projection-nonexpansive}, and continuity of the
principal square root on the positive-semidefinite cone, the fields
\(f_\ell\) and \(\widehat f_\ell^L\) are
\(\mathscr F_\ell\)-measurable. Moreover, for
\(\ell=0,\ldots,D-1\), \(z_{\ell+1}\) is independent of
\(\mathscr F_\ell\).

We use the Gaussian quadratic-form inequality
\cite[Proposition~1.1]{HsuKakadeZhang2012}: if
\(z\sim\mathcal N(0,I_p)\), \(\Sigma\succeq0\), and \(t>0\), then
\begin{equation}\label{eq:dgp-gaussian-qf}
    \Prob\!\left(
        z^\top\Sigma z
        \le
        \tr(\Sigma)+2\sqrt{\tr(\Sigma^2)t}
        +2\|\Sigma\|_{\op}t
    \right)
    \ge1-e^{-t}.
\end{equation}
For \(\varepsilon\in(0,1)\),  set
\(q_\varepsilon
    :=
    \left(
        1+2\sqrt{\log(1/\varepsilon)}+2\log(1/\varepsilon)
    \right)^{1/2}\).
We will use \eqref{eq:dgp-gaussian-qf} in the following form.
\begin{lemma}
\label{lem:dgp-weighted-gaussian-radius}
Let \(M\in\R^{m\times p}\) and \(z\sim\mathcal N(0,I_p)\). Then
\[
    \Prob\!\left(
        \|Mz\|_2\le q_\varepsilon\|M\|_F
    \right)
    \ge1-\varepsilon.
\]
In particular, if \(C\) is a positive semidefinite matrix with
\(\tr(D_rCD_r)\le T\), then
\[
    \Prob\!\left(
        \|D_rC^{1/2}z\|_2^2\le Tq_\varepsilon^2
    \right)
    \ge1-\varepsilon.
\]
\end{lemma}

\begin{proof}
Set \(\Sigma:=M^\top M\).  Then
\(z^\top\Sigma z=\|Mz\|_2^2\),
\(\tr(\Sigma)=\|M\|_F^2\),
\(\|\Sigma\|_{\op}\le\tr(\Sigma)\), and
\(\tr(\Sigma^2)\le\tr(\Sigma)^2\).  Taking
\(t=\log(1/\varepsilon)\) in \eqref{eq:dgp-gaussian-qf} proves the first
claim.  For the second take \(M=D_rC^{1/2}\), for which
\[
    \|M\|_F^2
    =\tr(D_rCD_r)
    \le T.
\]
\end{proof}
Using Lemma~\ref{lem:dgp-weighted-gaussian-radius} and the synthesis bound \eqref{eq:wavelet-synthesis-Hr}, define the deterministic radii recursively by
\begin{equation}\label{eq:dgp-radius-recursion}
    B_{0,\varepsilon}
    :=
    {S_{\rm Riesz}}^{1/2}\mathcal T_0^{1/2}q_\varepsilon,
    \qquad
    B_{\ell,\varepsilon}
    :=
    {S_{\rm Riesz}}^{1/2}
    \mathcal T_\ell(B_{\ell-1,\varepsilon})^{1/2}
    q_\varepsilon,
    \quad \ell=1,\dots,D-1.
\end{equation}
Define the events
\[
    \mathcal E_0
    :=
    \{\|f_0\|_{H^r}\le B_{0,\varepsilon}\},\qquad
    \mathcal E_\ell
    :=
    \mathcal E_{\ell-1}
    \cap
    \{\|f_\ell\|_{H^r}\le B_{\ell,\varepsilon}\},
    \quad \ell=1,\dots,D-1,
\]
which are nested, \(\mathcal E_{D-1}\subseteq\cdots\subseteq\mathcal E_0\), and \(\mathcal E_\ell\in\mathscr F_\ell\). 
\begin{proposition}\label{prop:dgp-good-event}
For every \(\ell=0,\dots,D-1\), \(\Prob(\mathcal E_\ell)\ge1-(\ell+1)\varepsilon\).
\end{proposition}
\begin{proof}
For \(\ell=0\), \(a_0=\mathbf C_0^{1/2}z_0\) and \(\tr(D_r\mathbf C_0D_r)\le\mathcal T_0\). Hence Lemma~\ref{lem:dgp-weighted-gaussian-radius} and \eqref{eq:wavelet-synthesis-Hr} give
\(
    \|f_0\|_{H^r}
    \le
    {S_{\rm Riesz}}^{1/2}\mathcal T_0^{1/2}q_\varepsilon
    =
    B_{0,\varepsilon}
\)
with probability at least \(1-\varepsilon\). 
\begingroup
\emergencystretch=2em
Thus \(\Prob(\mathcal E_0^c)\le\varepsilon\).
Let \(\ell\ge1\). On \(\mathcal E_{\ell-1}\), we have \(\|f_{\ell-1}\|_{H^r}\le B_{\ell-1,\varepsilon}\),
and hence, by \eqref{eq:dgp-trace-majorant}, \(
    \tr\!\left(
        D_r\CellJ{\ell}(f_{\ell-1})D_r
    \right)
    \le
    \mathcal T_\ell(B_{\ell-1,\varepsilon})\).
Conditionally on \(\mathscr F_{\ell-1}\), the covariance matrix above is
fixed and \(z_\ell\) is a standard Gaussian vector independent of
\(\mathscr F_{\ell-1}\). 
\par
\endgroup
Lemma~\ref{lem:dgp-weighted-gaussian-radius} and \eqref{eq:wavelet-synthesis-Hr} give
\[
    \Prob\!\left(
        \|f_\ell\|_{H^r}>B_{\ell,\varepsilon}
        \,\middle|\,
        \mathscr F_{\ell-1}
    \right)
    \le
    \varepsilon
    \qquad\text{on }\mathcal E_{\ell-1}.
\]
Since \(\mathcal E_{\ell-1}\in\mathscr F_{\ell-1}\), the law of total expectation
gives
\begin{align*}
    \Prob\!\left(
        \mathcal E_{\ell-1}
        \cap
        \{\|f_\ell\|_{H^r}>B_{\ell,\varepsilon}\}
    \right)&=
    \E\!\left[
        \mathbf 1_{\mathcal E_{\ell-1}}
        \Prob\!\left(
            \|f_\ell\|_{H^r}>B_{\ell,\varepsilon}
            \,\middle|\,
            \mathscr F_{\ell-1}
        \right)
    \right]\\&\le
\varepsilon\E\!\left[
    \mathbf 1_{\mathcal E_{\ell-1}}
\right]=
    \varepsilon\Prob(\mathcal E_{\ell-1})\le
    \varepsilon.
\end{align*}
Since \(
    \mathcal E_\ell^c
    =
    \mathcal E_{\ell-1}^c
    \cup
    \left(
        \mathcal E_{\ell-1}
        \cap\{\|f_\ell\|_{H^r}>B_{\ell,\varepsilon}\}
    \right)\),
we have
\(\Prob(\mathcal E_\ell^c)\le\Prob(\mathcal E_{\ell-1}^c)+\varepsilon\) and induction gives the result.
\end{proof}

\subsection{One-layer error bound}
\label{sec:dgp-one-step-final-error}
The goal is to establish an estimate of the form
\begin{equation}\label{eq:one-step-outline}
    e_\ell\le\alpha_\ell2^{-rL}+\gamma_\ell e_{\ell-1},
    \qquad
    e_\ell:=\|f_\ell-\widehat f_\ell^{\,L}\|_{L^2(\Omega)},
\end{equation}
where \(\alpha_\ell\) bounds the compression error at layer \(\ell\) and \(\gamma_\ell\) quantifies the amplification of the error inherited from layer \(\ell-1\).
The projected covariance may be singular, so we use an obvious extension to semidefinite matrices of the standard square-root perturbation bound (see, e.g.~\cite[Theorem~6.2]{Higham2008Functions}): 
If \(A\) and \(B\) are symmetric positive
semidefinite, then
\begin{equation}\label{eq:frobenius-square-root}
    \|A^{1/2}-B^{1/2}\|_F
    \le
    \frac{\|A-B\|_F}
    {\sqrt{\lambda_{\min}(A)}+\sqrt{\lambda_{\min}(B)}}.
\end{equation}
In particular, if \(A\succeq\kappa I\) for some \(\kappa>0\), then
\begin{equation}\label{eq:frobenius-square-root-one-gap}
    \|A^{1/2}-B^{1/2}\|_F
    \le
    \kappa^{-1/2}\|A-B\|_F.
\end{equation}
The constants in \eqref{eq:one-step-outline} are given by 
\[
    \alpha_0
    :=
    {S_{\rm Riesz}}^{1/2}
    \frac{\mathcal T_0}{\sqrt{\kappa_{0,J}}}
    q_\varepsilon,
\]
and for \(\ell=1,\dots,D\)
\[
    \alpha_\ell
    :=
    {S_{\rm Riesz}}^{1/2}
    \frac{\mathcal T_\ell(B_{\ell-1,\varepsilon})}
    {\sqrt{\kappa_{\ell,J}(B_{\ell-1,\varepsilon})}}
    q_\varepsilon,\qquad
    \gamma_\ell
    :=
    {S_{\rm Riesz}}^{1/2}
    \frac{\Lip_\ell}
    {\sqrt{\kappa_{\ell,J}(B_{\ell-1,\varepsilon})}}
    q_\varepsilon.
\]
Since the initial covariances are deterministic, a bound on \(e_0\) is easily obtained in the following lemma.
\begin{lemma}\label{lem:dgp-initial-projected-error}
For every \(L\), with probability at least \(1-\varepsilon\),
\begin{equation}\label{eq:dgp-initial-sample-error}
    e_0\le\alpha_02^{-rL}.
\end{equation}
\end{lemma}

\begin{proof}
By Lemma~\ref{lem:psd-projection-nonexpansive} and
\eqref{eq:dgp-uniform-tail-bound}, we have \(\|\mathbf C_0-\widehat{\mathbf C}_0^L\|_F\le\mathcal T_02^{-rL}\).
Moreover, \(\mathbf C_0\succeq\kappa_{0,J}I\) by Lemma~\ref{lem:dgp-finite-resolution-coercivity}, and \(\widehat{\mathbf C}_0^L\succeq0\) by construction.  Hence,~\eqref{eq:frobenius-square-root-one-gap} gives
\[
    \|\mathbf C_0^{1/2}-(\widehat{\mathbf C}_0^L)^{1/2}\|_F
    \le
    \frac{\mathcal T_0}{\sqrt{\kappa_{0,J}}}2^{-rL}.
\]
By definition,
\[
    a_0-\widehat a_0^L
    =
    \left(
        \mathbf C_0^{1/2}-(\widehat{\mathbf C}_0^L)^{1/2}
    \right)z_0.
\]
Thus the reconstruction formula \eqref{eq:dgp-finite-reconstruction} and the
\(L^2\)-synthesis estimate \eqref{eq:wavelet-synthesis-L2} give
\begin{align*}
    e_0^2
    &=
    \left\|
        \sum_{\iota\in\Lambda_J}
        (a_0-\widehat a_0^L)_\iota\widetilde\psi_\iota
    \right\|_{L^2(\Omega)}^2\le
    {S_{\rm Riesz}}
    \left\|
        \left(
            \mathbf C_0^{1/2}-(\widehat{\mathbf C}_0^L)^{1/2}
        \right)z_0
    \right\|_2^2.
\end{align*}
The first claim of
Lemma~\ref{lem:dgp-weighted-gaussian-radius}, applied with
\(M=\mathbf C_0^{1/2}-(\widehat{\mathbf C}_0^L)^{1/2}\), gives, with
probability at least \(1-\varepsilon\),
\[
    e_0
    \le
    {S_{\rm Riesz}}^{1/2}q_\varepsilon
    \|\mathbf C_0^{1/2}-(\widehat{\mathbf C}_0^L)^{1/2}\|_F
    \le
    \alpha_02^{-rL}.
\]
\end{proof}
For \(\ell\ge1\), the covariance at layer \(\ell\) is fixed given \(\mathscr F_{\ell-1}\), while \(z_\ell\) is independent of \(\mathscr F_{\ell-1}\), so the one-step estimate is stated conditionally.
\begin{proposition}\label{prop:dgp-one-step-minimal}
Let \(\ell\in\{1,\dots,D\}\). Then, for every \(L\),
\begin{equation}\label{eq:dgp-one-step-minimal}
    \Prob\!\left(
        e_\ell
        \le
        \alpha_\ell2^{-rL}+\gamma_\ell e_{\ell-1}
        \,\middle|\,
        \mathscr F_{\ell-1}
    \right)
    \ge
    1-\varepsilon
    \qquad\text{on }\mathcal E_{\ell-1}.
\end{equation}
\end{proposition}

\begin{proof}
Set \(f:=f_{\ell-1}\) and \(g:=\widehat f_{\ell-1}^L\).
On \(\mathcal E_{\ell-1}\),
\(\|f\|_{H^r}\le B_{\ell-1,\varepsilon}\).  Since
\(\mathbf C_\ell(f)\succeq0\), we have
\(\Pi_+(\mathbf C_\ell(f))=\mathbf C_\ell(f)\).  Therefore
Lemma~\ref{lem:psd-projection-nonexpansive} gives
\begin{align*}
\|\mathbf C_\ell(f)-\widehat{\mathbf C}_\ell^L(g)\|_F
&\le
\|\mathbf C_\ell(f)-\mathbf C_\ell^L(g)\|_F\le
\|\mathbf C_\ell(f)-\mathbf C_\ell^L(f)\|_F\\&
+
\|\mathbf C_\ell^L(f)-\mathbf C_\ell^L(g)\|_F\le
\mathcal T_\ell(B_{\ell-1,\varepsilon})2^{-rL}
+
\Lip_\ell e_{\ell-1},
\end{align*}
where the last step uses \eqref{eq:dgp-uniform-tail-bound} and \eqref{eq:dgp-global-frob-lipschitz}.
By Lemma~\ref{lem:dgp-finite-resolution-coercivity},
\[
    \mathbf C_\ell(f)
    \succeq
    \kappa_{\ell,J}(B_{\ell-1,\varepsilon})I,
\]
whereas \(\widehat{\mathbf C}_\ell^L(g)\succeq0\).  Hence,~\eqref{eq:frobenius-square-root-one-gap} gives
\begin{equation}\label{eq:dgp-projected-square-root-one-step}
\|\mathbf C_\ell(f)^{1/2}
-\widehat{\mathbf C}_\ell^L(g)^{1/2}\|_F\le
\frac{
\mathcal T_\ell(B_{\ell-1,\varepsilon})2^{-rL}
+
\Lip_\ell e_{\ell-1}}{\sqrt{\kappa_{\ell,J}(B_{\ell-1,\varepsilon})}}.
\end{equation}
Since \(f\) and \(g\) are \(\mathscr F_{\ell-1}\)-measurable, the matrix \(\mathbf C_\ell(f)^{1/2}-\widehat{\mathbf C}_\ell^L(g)^{1/2}\) is fixed conditionally on \(\mathscr F_{\ell-1}\), while \(z_\ell\) remains an independent standard Gaussian vector. 
Moreover,
\[
    a_\ell-\widehat a_\ell^L
    =
    \left(
        \mathbf C_\ell(f)^{1/2}
        -\widehat{\mathbf C}_\ell^L(g)^{1/2}
    \right)z_\ell.
\]
Hence \eqref{eq:dgp-finite-reconstruction} and
\eqref{eq:wavelet-synthesis-L2} give
\begin{align*}
    e_\ell^2
    &=
    \left\|
        \sum_{\iota\in\Lambda_J}
        (a_\ell-\widehat a_\ell^L)_\iota\widetilde\psi_\iota
    \right\|_{L^2(\Omega)}^2
    \le
    {S_{\rm Riesz}}
    \left\|
        \left(
            \mathbf C_\ell(f)^{1/2}
            -\widehat{\mathbf C}_\ell^L(g)^{1/2}
        \right)z_\ell
    \right\|_2^2.
\end{align*}
The first claim of Lemma~\ref{lem:dgp-weighted-gaussian-radius}, applied
conditionally with
\(M=\mathbf C_\ell(f)^{1/2}-\widehat{\mathbf C}_\ell^L(g)^{1/2}\), therefore
gives, with conditional probability at least \(1-\varepsilon\),
\[
    e_\ell
    \le
    {S_{\rm Riesz}}^{1/2}q_\varepsilon
    \|\mathbf C_\ell(f)^{1/2}
      -\widehat{\mathbf C}_\ell^L(g)^{1/2}\|_F.
\]
Substituting \eqref{eq:dgp-projected-square-root-one-step} and the definitions
of \(\alpha_\ell\) and \(\gamma_\ell\) proves
\eqref{eq:dgp-one-step-minimal}.
\end{proof}

\subsection{Proof of the main theorem}\label{sec:main_thm_proof}
The proof applies Gaussian concentration \(2D+1\) times, so the radii \(B_{\ell,\varepsilon}\) of \eqref{eq:dgp-radius-recursion} and the coefficients \(\alpha_\ell,\gamma_\ell\) of Section~\ref{sec:dgp-one-step-final-error} are evaluated at \(\varepsilon:=\eta/(2D+1)\), where \(\eta\) is the desired failure probability in Theorem~\ref{thm:dgp-total-error-layer-D}.

Define
\begin{equation}\label{eq:dgp-propagation-prefactor}
P_0:=\alpha_0,
\qquad
P_\ell:=\alpha_\ell+\gamma_\ell P_{\ell-1},
\qquad
\ell=1,\dots,D,
\end{equation}
so that \(P_\ell\) sums the compression errors from layers \(0,\dots,\ell\), each amplified by the later coefficients \(\gamma\), and
\begin{equation}\label{eq:dgp-terminal-cov-prefactor}
K_D^{\mathrm{cov}}
:=
\mathcal T_D(B_{D-1,\varepsilon})
+\Lip_D P_{D-1}.
\end{equation}

\begin{proof}[Proof of Theorem~\ref{thm:dgp-total-error-layer-D}]
Fix \(L\in\N\). By Proposition~\ref{prop:dgp-good-event}, \(\Prob(\mathcal E_{D-1}^c)\le D\varepsilon\).
Let \(\mathcal G_0\) denote the event \(\{e_0\le\alpha_02^{-rL}\}\) in
\eqref{eq:dgp-initial-sample-error}. Then
Lemma~\ref{lem:dgp-initial-projected-error} gives \(\Prob(\mathcal G_0^c)\le\varepsilon\).
For \(\ell=1,\dots,D\), let \(\mathcal G_\ell\) denote the event \(\{e_\ell\le\alpha_\ell2^{-rL}+\gamma_\ell e_{\ell-1}\}\) in
\eqref{eq:dgp-one-step-minimal}. By
Proposition~\ref{prop:dgp-one-step-minimal},
\[
    \Prob(
        \mathcal G_\ell^c
        \mid\mathscr F_{\ell-1}
    )
    \le\varepsilon
    \qquad\text{on }\mathcal E_{\ell-1}.
\]
Since
\(\mathcal E_{D-1}\subseteq\mathcal E_{\ell-1}\) and
\(\mathcal E_{\ell-1}\in\mathscr F_{\ell-1}\), the law of total expectation gives
\[
\Prob(\mathcal E_{\ell-1}\cap \mathcal G_\ell^c)
=
\E\!\left[
    \mathbf 1_{\mathcal E_{\ell-1}}
    \Prob\!\left(
        \mathcal G_\ell^c
        \,\middle|\,
        \mathscr F_{\ell-1}
    \right)
\right]
\le
\varepsilon
\E\!\left[\mathbf 1_{\mathcal E_{\ell-1}}\right]
=
\varepsilon\Prob(\mathcal E_{\ell-1})
\le\varepsilon
\]
and because \(\mathcal E_{D-1}\subseteq\mathcal E_{\ell-1}\), it follows that \(\Prob(\mathcal E_{D-1}\cap \mathcal G_\ell^c)\le\varepsilon\).
A union bound over the one-step failures therefore gives
\[
    \Prob\!\left(
        \mathcal E_{D-1}
        \cap
        \bigcup_{\ell=1}^{D}\mathcal G_\ell^c
    \right)
    \le
    D\varepsilon.
\]
  Therefore, with
  \(\mathcal H:=\mathcal E_{D-1}\cap \mathcal G_0\cap\bigcap_{\ell=1}^{D}\mathcal G_\ell\),
  a union bound gives
  \begin{equation}\label{eq:dgp-G-failure}
      \Prob(\mathcal H^c)
      \le
      D\varepsilon+\varepsilon+D\varepsilon
      =(2D+1)\varepsilon
      =\eta.
  \end{equation}
On \(\mathcal H\), \(e_0\le\alpha_02^{-rL}=P_02^{-rL}\) holds.
If \(e_{\ell-1}\le P_{\ell-1}2^{-rL}\), then
\eqref{eq:dgp-one-step-minimal} and
\eqref{eq:dgp-propagation-prefactor} give
\[
    e_\ell
    \le
    \alpha_\ell2^{-rL}+\gamma_\ell e_{\ell-1}
    \le
    (\alpha_\ell+\gamma_\ell P_{\ell-1})2^{-rL}
    =P_\ell2^{-rL}.
\]
Thus, on \(\mathcal H\), \(e_\ell\le P_\ell2^{-rL}\) for every \(\ell=0,\dots,D\).

For the terminal covariance, note that \(\|f_{D-1}\|_{H^r}\le B_{D-1,\varepsilon}\) on \(\mathcal H\subseteq\mathcal E_{D-1}\). Arguing as in the proof of Proposition~\ref{prop:dgp-one-step-minimal}, nonexpansiveness of \(\Pi_+\) (Lemma~\ref{lem:psd-projection-nonexpansive}), the bounds \eqref{eq:dgp-uniform-tail-bound} and \eqref{eq:dgp-global-frob-lipschitz}, and the bound on \(e_{D-1}\) just proved give, on \(\mathcal H\),
\begin{align*}
\|\CellJ{D}(f_{D-1})
  -\widehat{\mathbf C}_D^L(\widehat f_{D-1}^L)\|_F
&\le
\|\CellJ{D}(f_{D-1})-\CellJL{D}(\widehat f_{D-1}^L)\|_F\\
&\quad\le
\mathcal T_D(B_{D-1,\varepsilon})2^{-rL}
+\Lip_D\|f_{D-1}-\widehat f_{D-1}^L\|_{L^2}\\
&\quad\le
\left(
    \mathcal T_D(B_{D-1,\varepsilon})+\Lip_DP_{D-1}
\right)2^{-rL}
=K_D^{\mathrm{cov}}2^{-rL}.
\end{align*}
This proves \eqref{eq:dgp-final-covariance-error}. Hence all assertions of the theorem hold on \(\mathcal H\), which has probability at least \(1-\eta\) by \eqref{eq:dgp-G-failure}.
\end{proof}

\section{Numerical experiments}\label{sec:numerics}
All numerical experiments are performed in one spatial dimension, on \(\Omega=[0,1]\).
We construct Symlet boundary-adapted wavelets from~\cite[Section~4]{cohen1993wavelets}. 
Starting from the low- and high-pass filters associated with a scaling function and a mother wavelet with \(N\) vanishing moments, respectively, they propose an orthonormalisation procedure that produces \(N\) left boundary functions, \(N\) right boundary functions, and \(2^j-2N\) interior functions at resolution level \(j\in\mathbb{N}\). As shown in~\cite[Theorem~4.4]{cohen1993wavelets}, the scaling functions at a fixed coarsest level \(j_0\), together with the boundary-adapted wavelets over all levels \(j\geq j_0\), form an orthonormal basis of \(L^2([0,1])\). The coarsest level \(j_0\) is chosen such that \(2^{j_0}\geq 2N\). At each level \(j=j_0,\ldots,J\), the construction yields \(2^j\) wavelet functions, so that the total number of wavelets is \(2^{J+1}-2^{j_0}\). Combined with the \(2^{j_0}\) scaling functions at the coarsest level, the total number of basis functions at level \(J\) is \(2^{J+1}\). The orthonormalisation coefficients are determined solely by the underlying filter sequence and are therefore independent of the resolution level \(j\). Since the boundary functions are finite linear combinations with scale-independent coefficients (the right boundary functions are additionally obtained by shifts, which do not affect Sobolev regularity), estimate \eqref{eq:wavelet-synthesis-Hr} holds with constants uniform in \(j\) on the unit interval. We refer to \cite[Section~4]{cohen1993wavelets} for further details on the construction.

For the numerical realisation of this construction, we use the filter and orthonormalisation coefficients provided in \texttt{IntervalWavelets.jl} \cite{jensen_intervalwavelets}. The scaling functions are evaluated on a dyadic grid using the direct approach in \cite[Section~7.2]{daubechies1992ten}. Their values at integer grid points are obtained from the eigenvector associated with eigenvalue \(1\) of the corresponding refinement matrix, and the refinement equation is subsequently applied recursively to determine their values on a uniformly spaced dyadic grid. This grid is then used directly in a composite Simpson quadrature rule to approximate the entries of the wavelet covariance matrices.

For sampling, recall that the dense matrix \(\widehat{\mathbf C}_0^L\) need not be assembled explicitly. Indeed, using a Lanczos/Krylov subspace procedure the square-root action \((\widehat{\mathbf C}_0^L)^{\frac12} z\) can be approximated. More precisely, we approximate the action of the matrix function \(h({\mathbf C}^L)\) on \(z\), where \(h(x):=\sqrt{\max(x,0)}\).
After \(m\) Lanczos iterations with starting vector \(q_1=z/\|z\|_2\), we obtain an orthonormal basis \(Q_m\) of the Krylov subspace together with the symmetric tridiagonal matrix
\[
T_m
=
Q_m^\top \mathbf C^L Q_m.
\]
The action of the matrix function is then approximated by
\[
h(\mathbf C^L)z
\approx
\|z\|_2\,Q_m h(T_m)e_1.
\]
Since \(T_m\in\mathbb R^{m\times m}\) is symmetric tridiagonal, \(h(T_m)\) can be evaluated directly via its spectral decomposition, see for example \cite{druskin1998lanczos} for further details.
\begin{figure}[htbp!]
    \centering

    \begin{subfigure}{0.46\textwidth}
        \centering
        \includegraphics[
            width=\linewidth,
            height=0.3\textheight,
            keepaspectratio
        ]{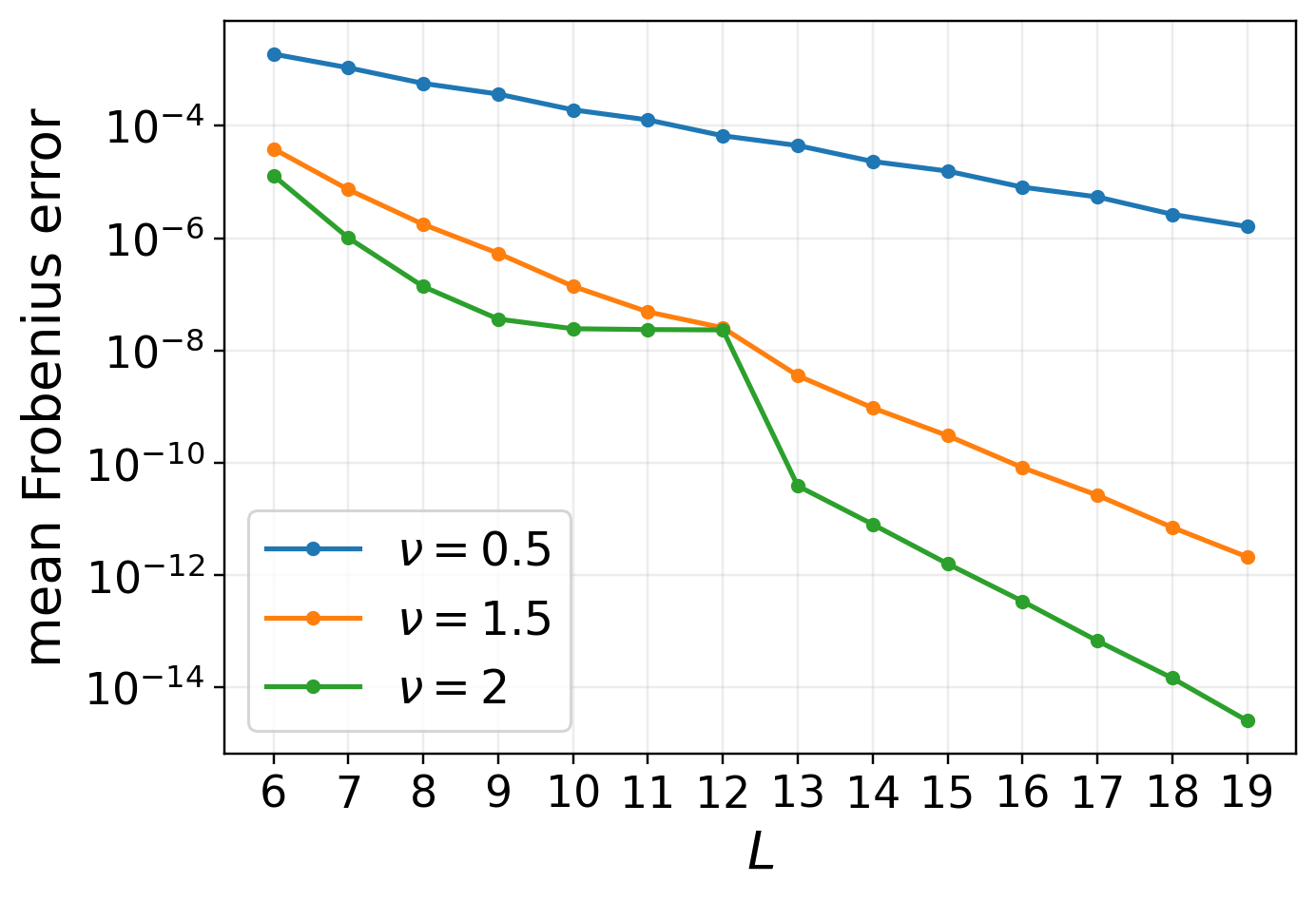}
        \caption{Sym4}
        \label{fig:initial-square}
    \end{subfigure}
    \hfill
    \begin{subfigure}{0.46\textwidth}
        \centering
        \includegraphics[
            width=\linewidth,
            height=0.3\textheight,
            keepaspectratio
        ]{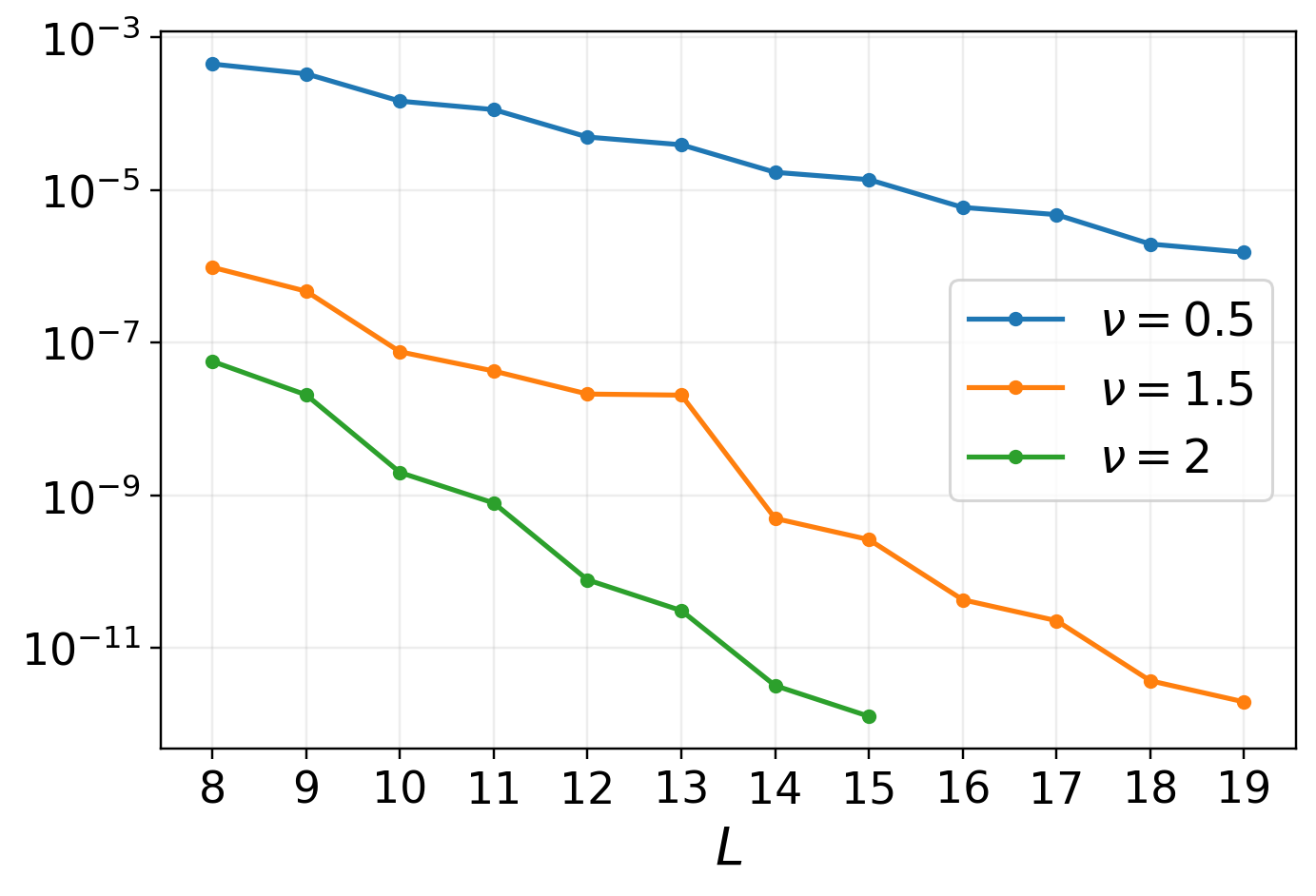}
        \caption{Sym8}
        \label{fig:initial-exp}
    \end{subfigure}

    \caption{Initial-layer covariance error \(\|\bC_0-\bC_0^L\|_F\) versus \(L\).}
    \label{fig:initial-layer}
\end{figure}

\begin{figure}[htbp]
    \centering
    \includegraphics[width=0.9\textwidth, height=0.3\textheight]{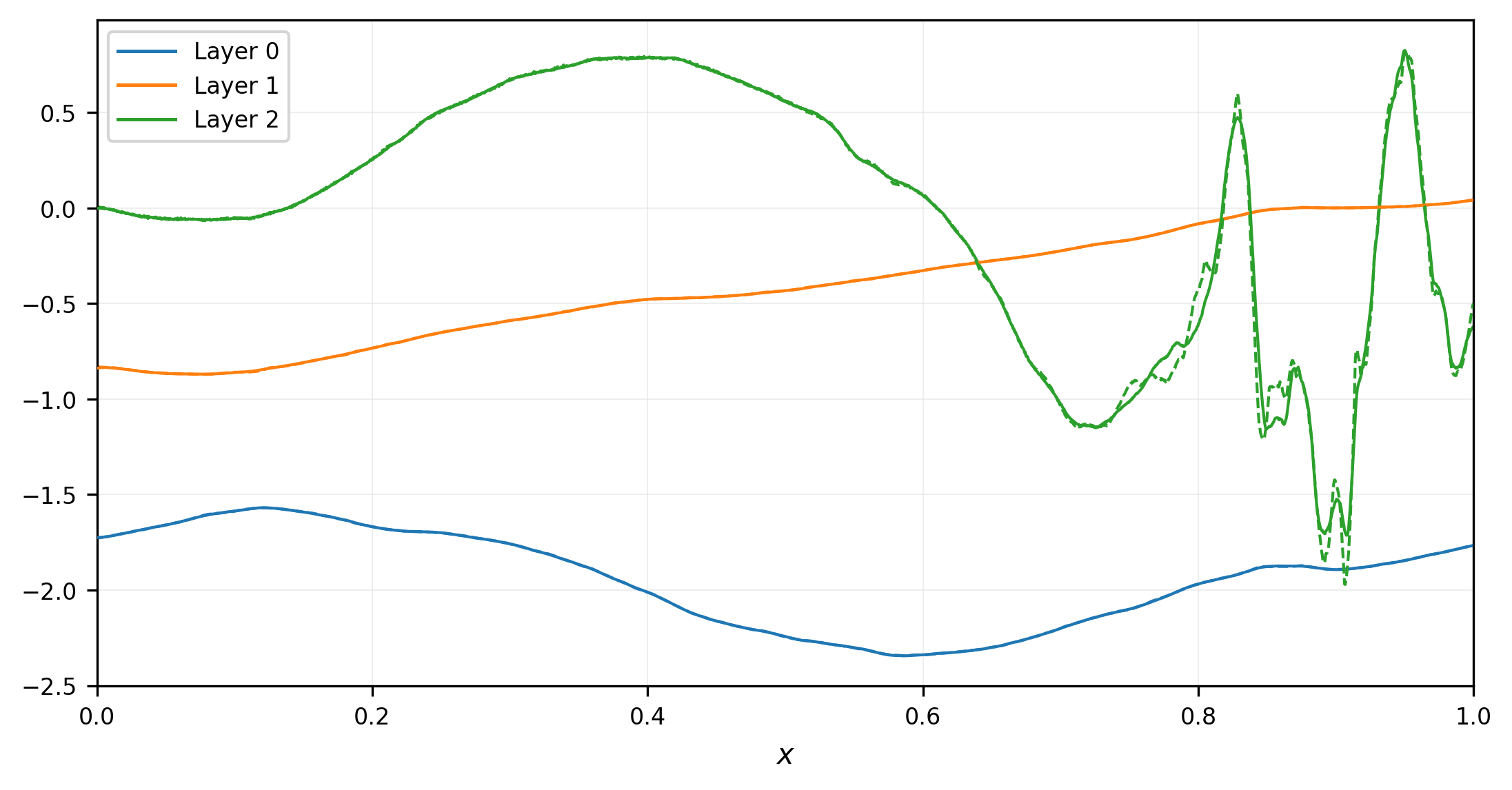}
    \caption{Samples paths in Sym4 wavelet basis for \(\nu=1.5, J=10\). The dashed lines are the truncated samples with \(L=9\) (no visible difference for layers 0 and 1).}
    \label{fig:samples}
\end{figure}
We use  Sym4 and Sym8 wavelets (Symlets with $N=4,8$ vanishing moments) in the experiments below. The resolution is fixed at $J=10$ (except for Sym8 with \(\nu=2\), where we set \(J=8\)), resulting in \(2048\) (and \(512\)) coefficients respectively. The coarsest levels are \(j_0=3\) and \(j_0=4\) for Sym4 and Sym8 respectively. 
 Every DGP in the experiments starts with an initial layer using variance \(\sigma^2=1\) and correlation length \(\lambda_0=1\).
For quadrature, we use a grid with stepsize \(2^{-16}\) and we sample with \(60\) Lanczos iterations. The errors in Figure~\ref{fig:combined-wavelets} are averages
of Frobenius covariance errors over \(15\) independent realisations. At each layer we calculate the Frobenius covariance error before projecting the truncated covariance. Lemma~\ref{lem:psd-projection-nonexpansive} ensures this is an upper bound on \(\|\bC_{\ell}-\widehat{\bC}^L_{\ell}\|_F\).

Figure~\ref{fig:samples} shows a sample of the exact DGP in wavelet basis and its truncated approximation, while Figure~\ref{fig:initial-layer} shows the initial-layer covariance errors using boundary-adapted Sym4 and Sym8 wavelets. For both families, the errors decrease rapidly with \(L\), with faster decay for smoother Mat\'ern kernels. This behaviour is qualitatively consistent with the regularity-dependent truncation bounds in Theorem~\ref{thm:trace-controlled-scale-tail} and Lemma~\ref{lem:conv-matern-sobolev-trace}. The Sym8 experiment with \(\nu=2\) uses the coarser resolution \(J=8\), so its decay cannot be attributed to kernel smoothness alone.

Figure~\ref{fig:combined-wavelets} shows the corresponding behaviour through three subsequent layers, using \(F(t)=t^2+10^{-3}\) and \(F(t)=e^t\). 
The errors generally increase with depth. This is likely due to the fact that the quadratic link produces short correlation lengths near zeros of the preceding field, and is more sensitive to relative perturbations. Indeed, \(L_F=\sup_t|F'(t)/F(t)|\) equals \(\sqrt{1000}\) for the quadratic link and \(1\) for the exponential link. This also provides an explanation for the fact that the separation between layers is more pronounced for the quadratic link.

The errors also decrease for \(\nu=0.5\), although this case lies outside the regularity assumptions of Theorem~\ref{thm:dgp-total-error-layer-D}.

\begin{figure}[p]
    \centering
    \captionsetup[subfigure]{font=small,skip=2pt}

    \begin{minipage}{0.32\textwidth}
        \centering
        \textbf{\(\nu=0.5\)}
    \end{minipage}
    \hfill
    \begin{minipage}{0.32\textwidth}
        \centering
        \textbf{\(\nu=1.5\)}
    \end{minipage}
    \hfill
    \begin{minipage}{0.32\textwidth}
        \centering
        \textbf{\(\nu=2\)}
    \end{minipage}

    \par\vspace{1mm}

    \begin{subfigure}{\textwidth}
        \centering
        \includegraphics[
            width=0.325\linewidth,
            height=0.20\textheight,
            keepaspectratio
        ]{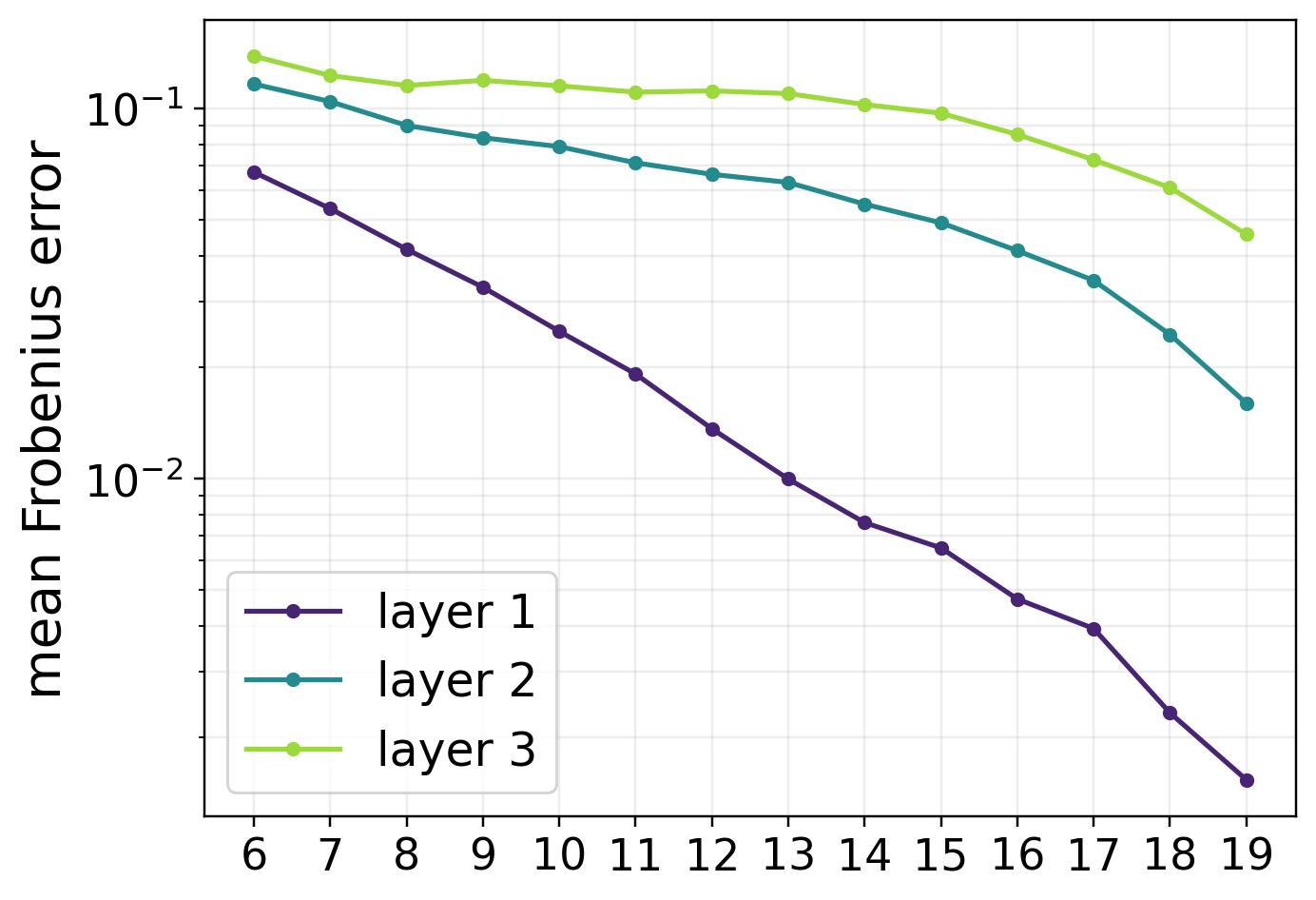}
        \hfill
        \includegraphics[
            width=0.325\linewidth,
            height=0.20\textheight,
            keepaspectratio
        ]{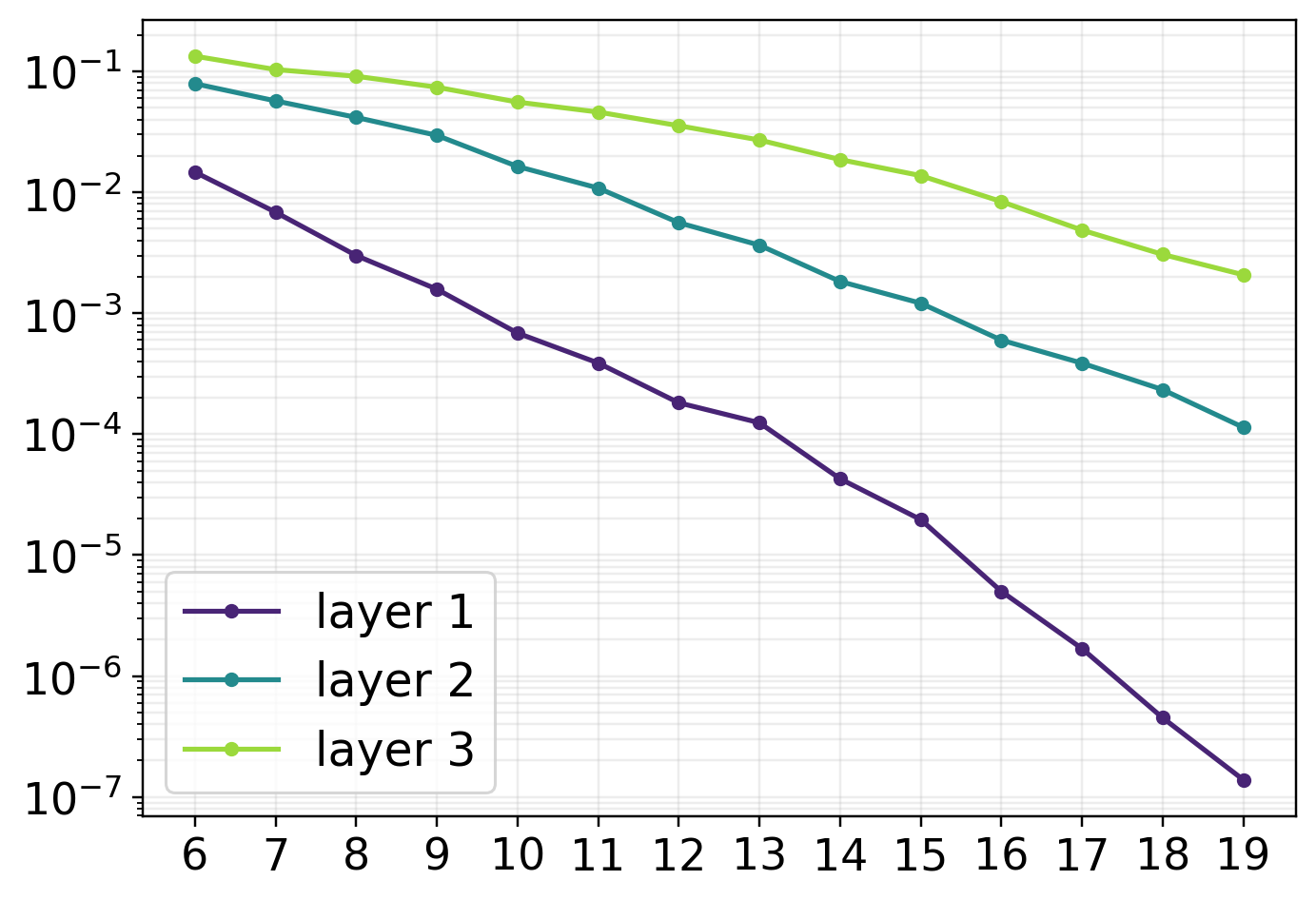}
        \hfill
        \includegraphics[
            width=0.325\linewidth,
            height=0.20\textheight,
            keepaspectratio
        ]{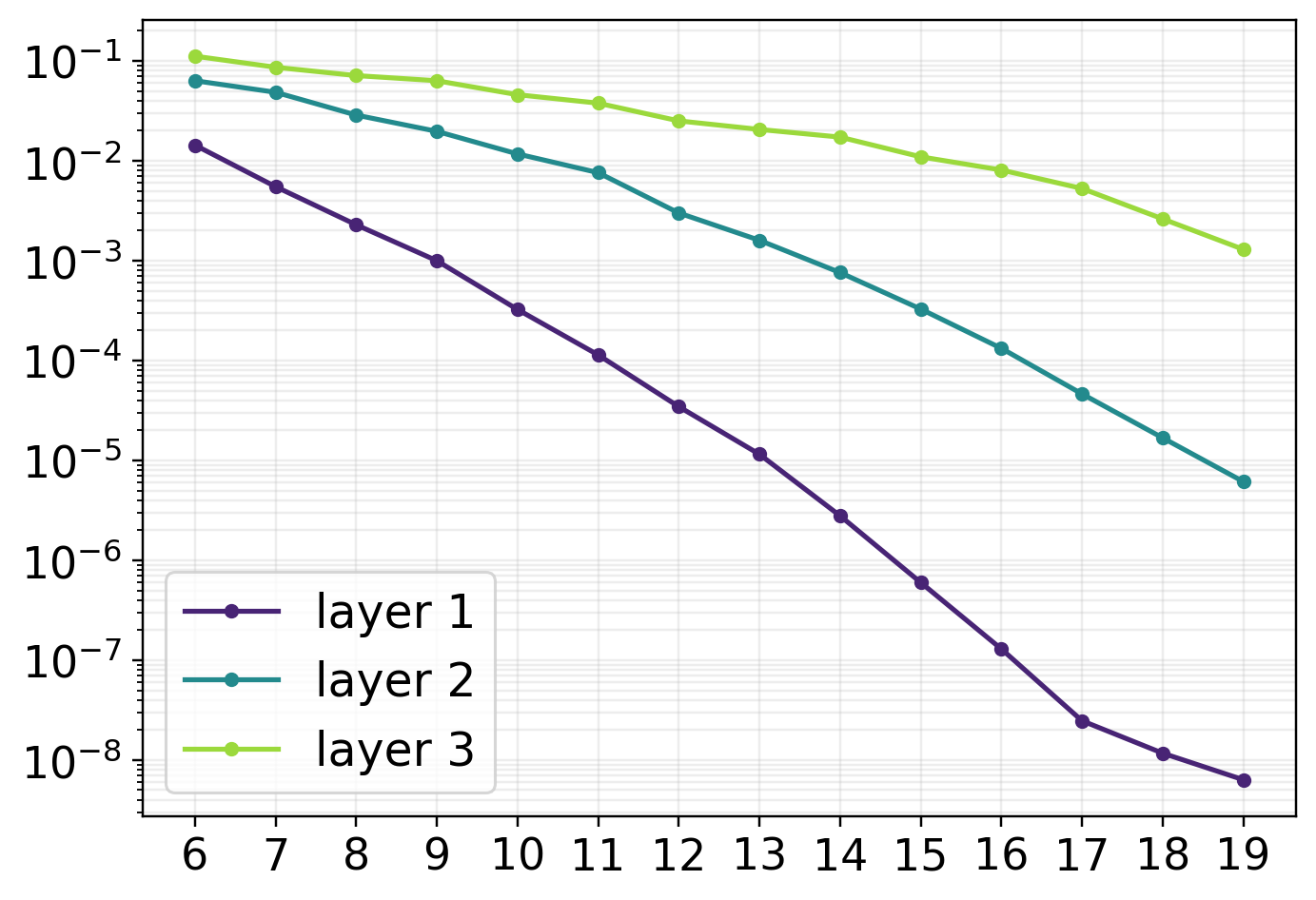}

        \caption{Sym4: \(F(x)=x^2+10^{-3}\)}
        \label{fig:sym4-square}
    \end{subfigure}

    \par\vspace{1mm}

    \begin{subfigure}{\textwidth}
        \centering
        \includegraphics[
            width=0.325\linewidth,
            height=0.20\textheight,
            keepaspectratio
        ]{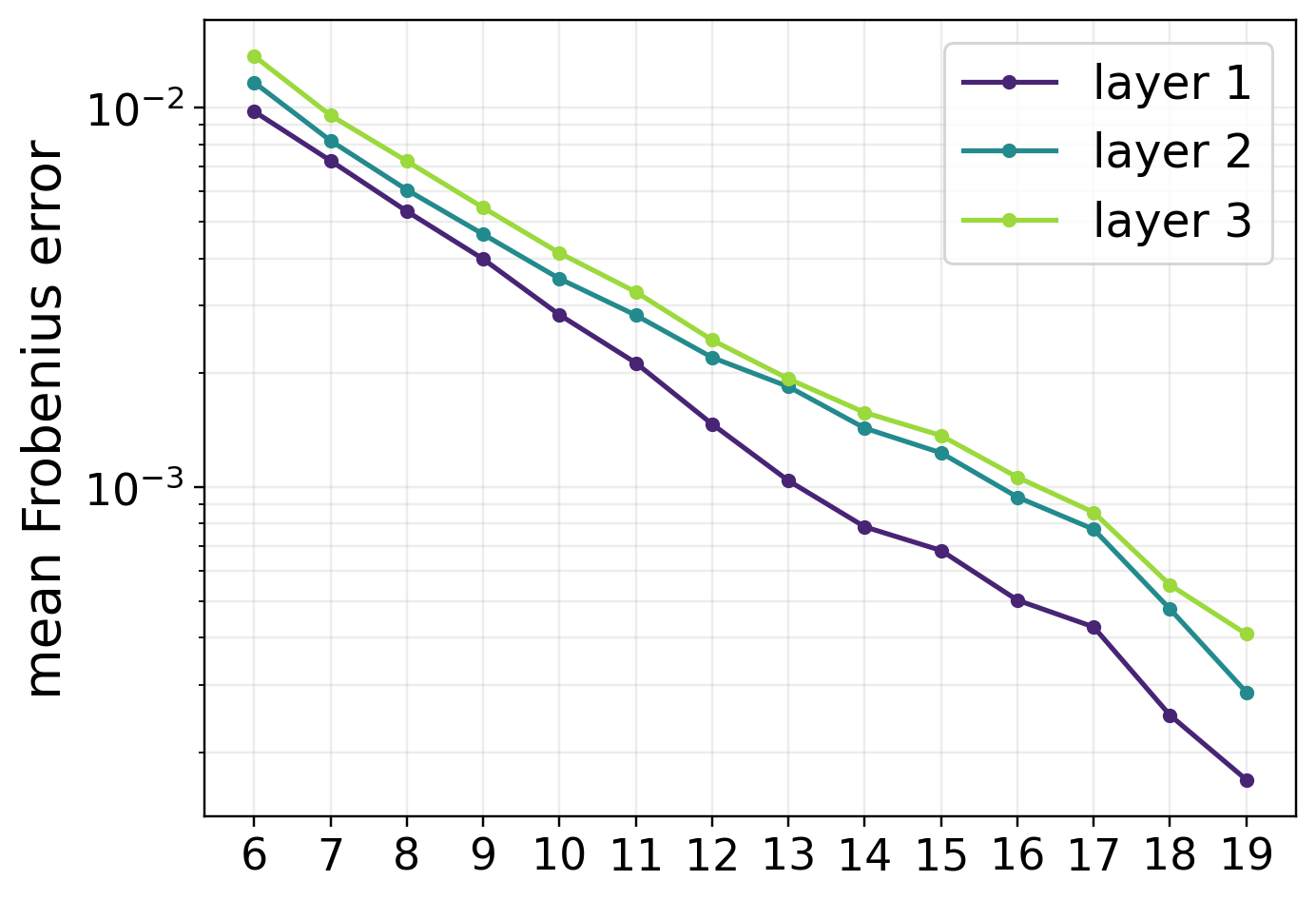}
        \hfill
        \includegraphics[
            width=0.325\linewidth,
            height=0.20\textheight,
            keepaspectratio
        ]{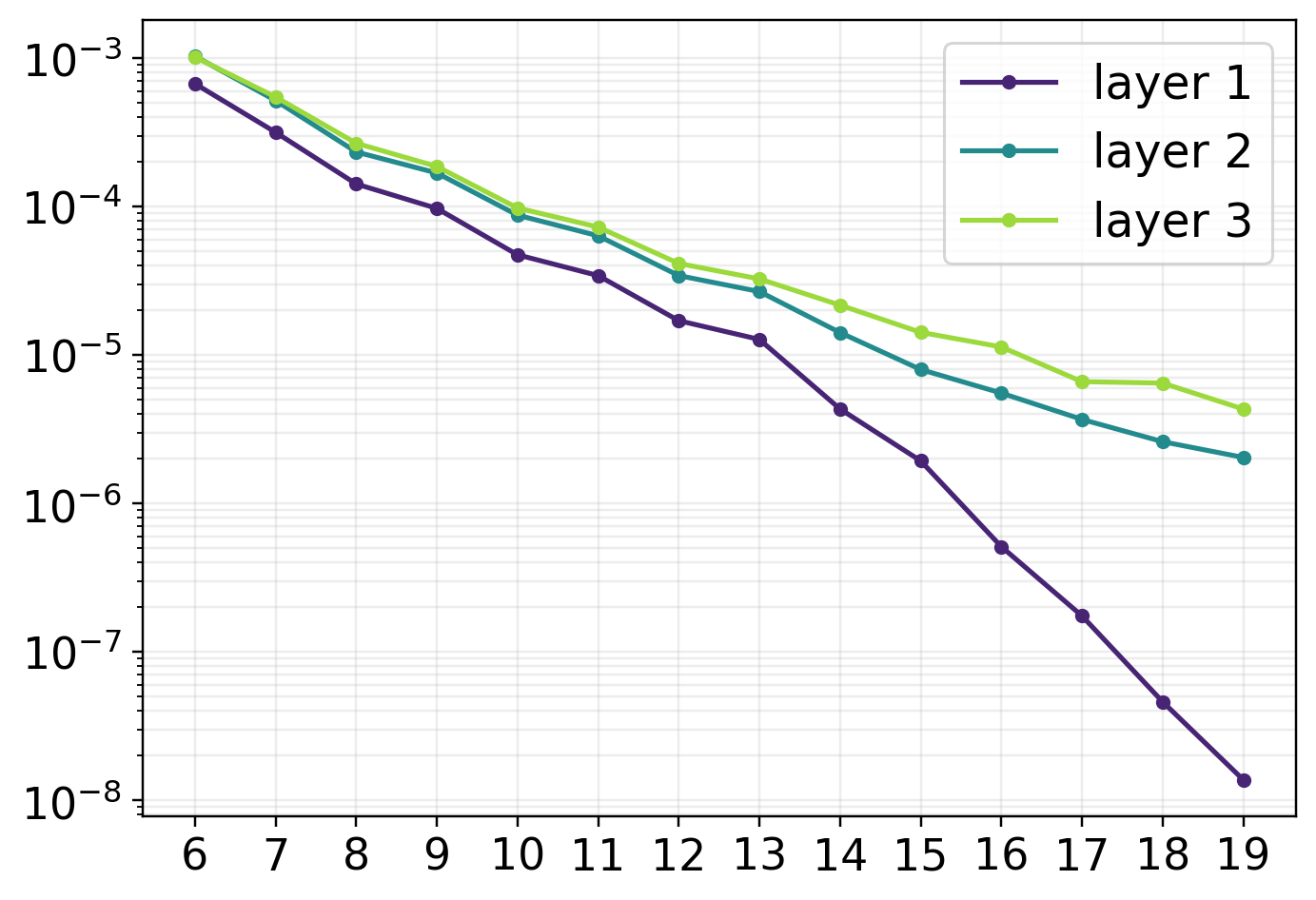}
        \hfill
        \includegraphics[
            width=0.325\linewidth,
            height=0.20\textheight,
            keepaspectratio
        ]{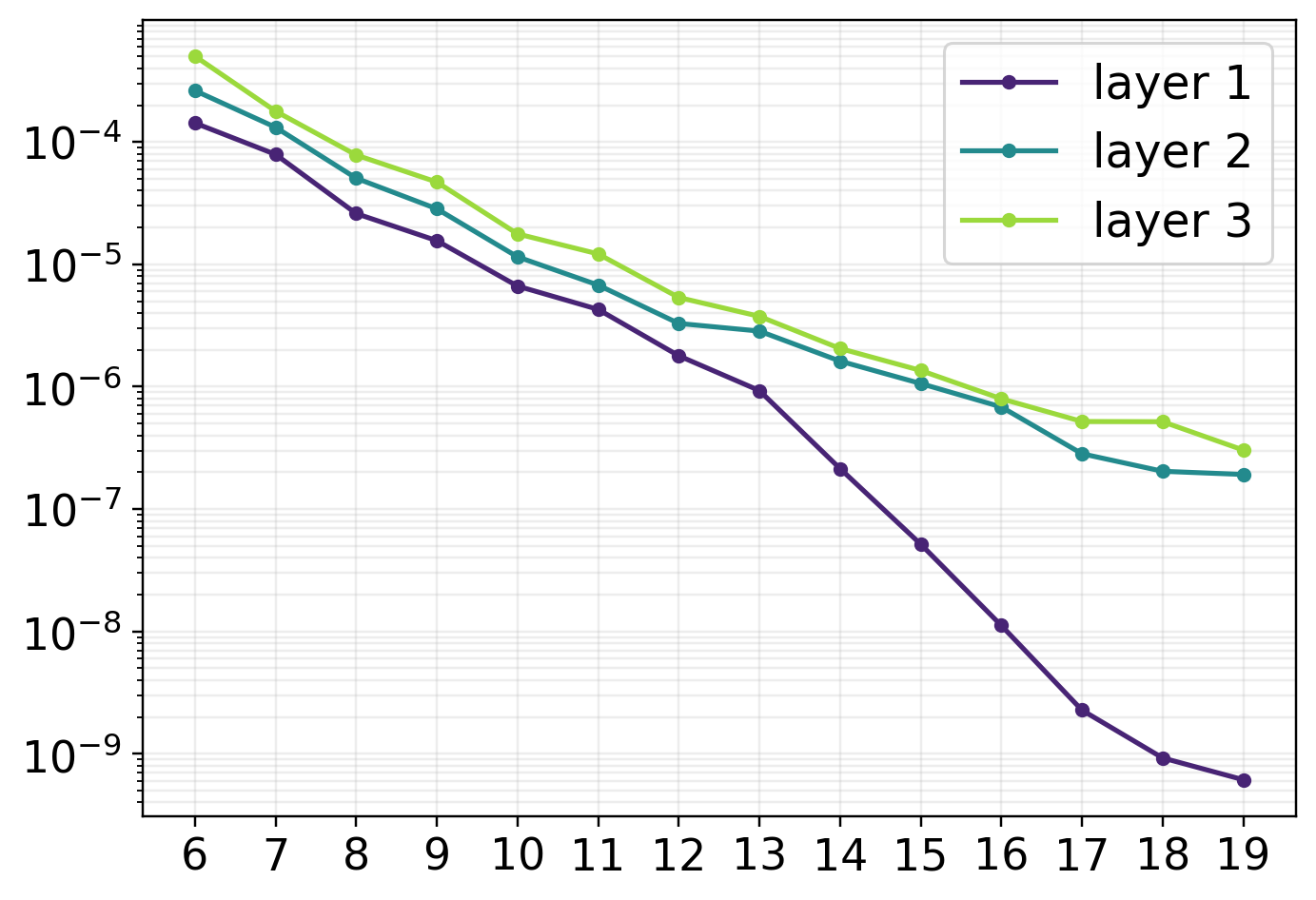}

        \caption{Sym4: \(F(x)=\exp(x)\)}
        \label{fig:sym4-exp}
    \end{subfigure}

    \par\vspace{1mm}

    \begin{subfigure}{\textwidth}
        \centering
        \includegraphics[
            width=0.325\linewidth,
            height=0.20\textheight,
            keepaspectratio
        ]{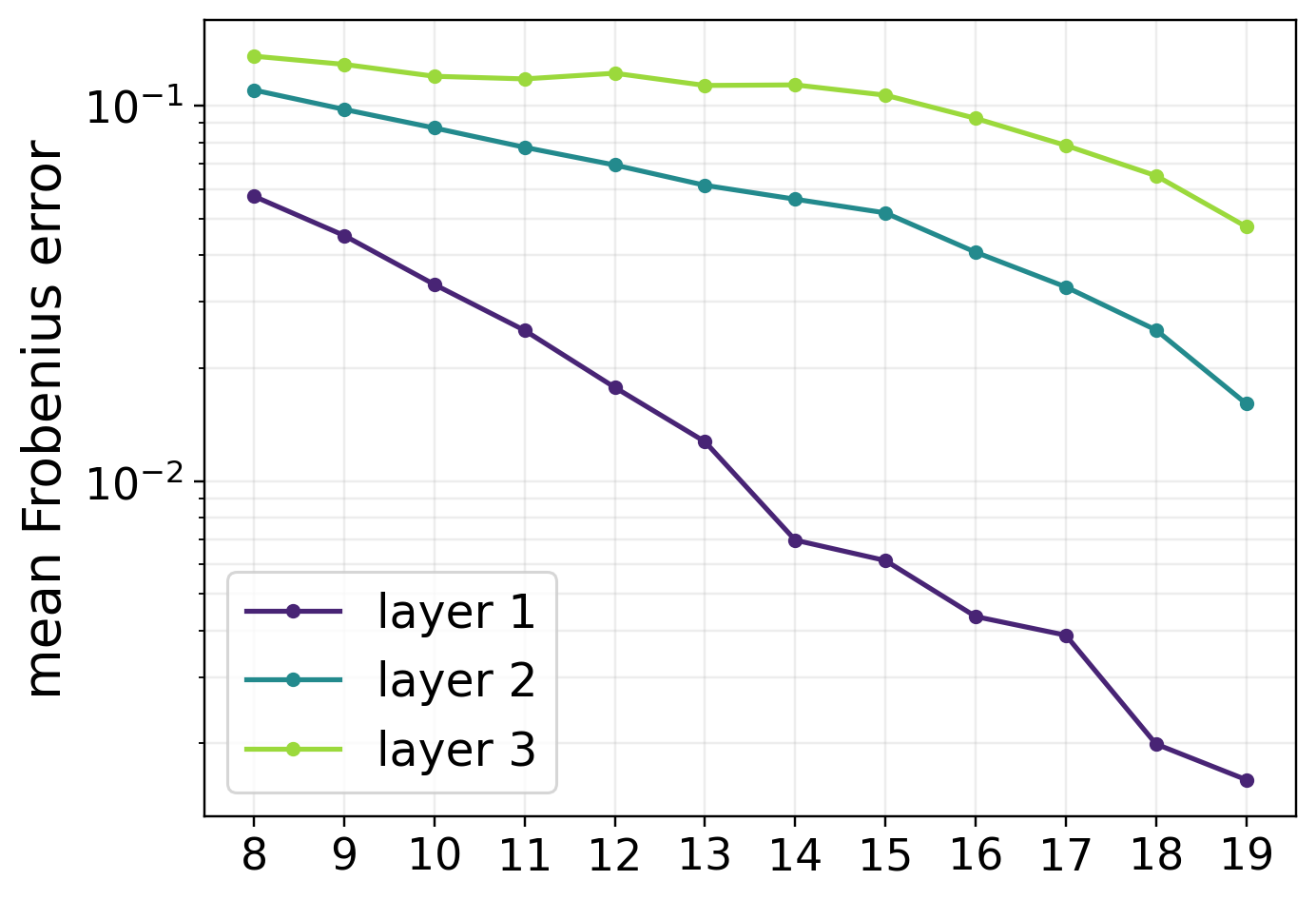}
        \hfill
        \includegraphics[
            width=0.325\linewidth,
            height=0.20\textheight,
            keepaspectratio
        ]{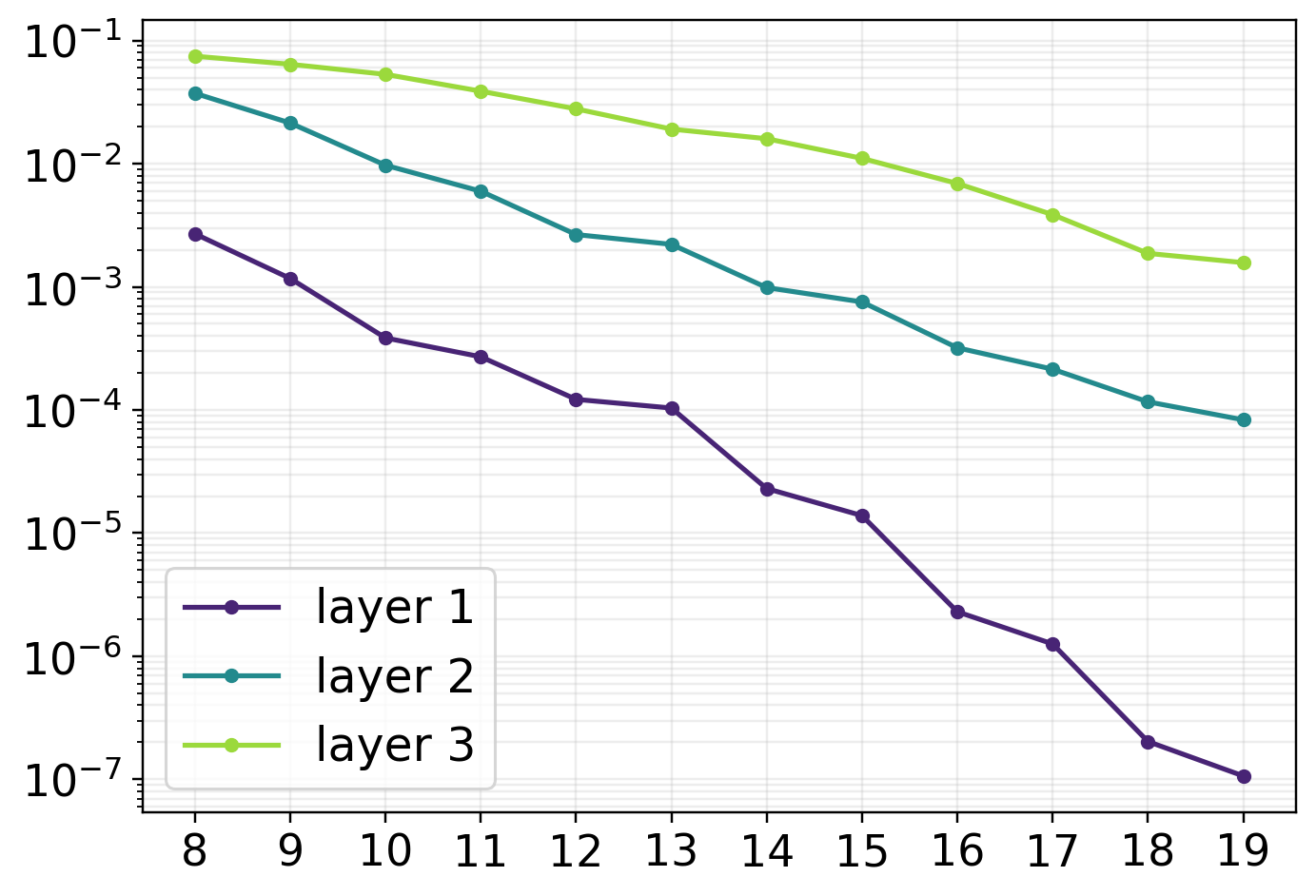}
        \hfill
        \includegraphics[
            width=0.325\linewidth,
            height=0.20\textheight,
            keepaspectratio
        ]{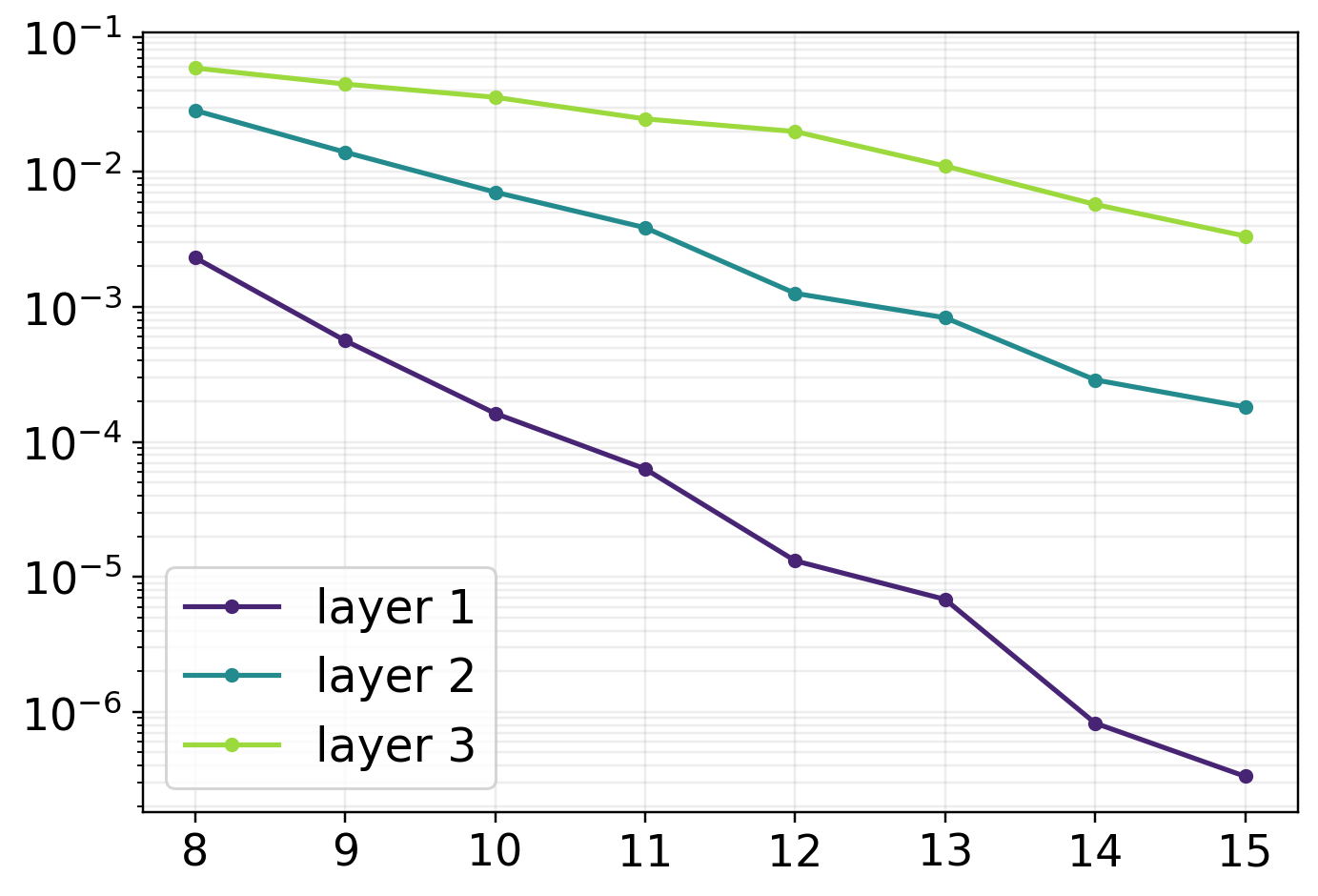}

        \caption{Sym8: \(F(x)=x^2+10^{-3}\)}
        \label{fig:sym8-square}
    \end{subfigure}

    \par\vspace{1mm}

    \begin{subfigure}{\textwidth}
        \centering
        \includegraphics[
            width=0.325\linewidth,
            height=0.20\textheight,
            keepaspectratio
        ]{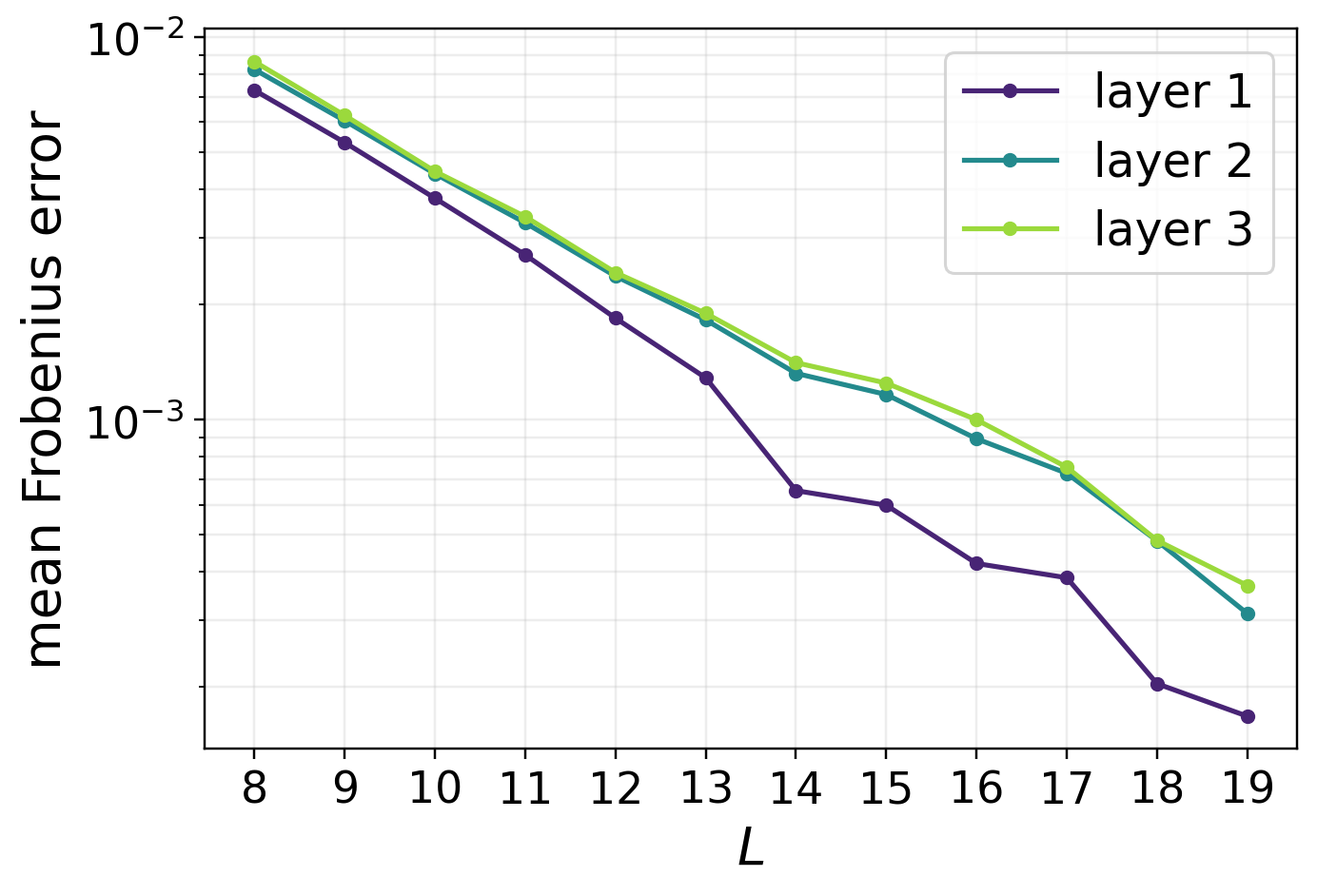}
        \hfill
        \includegraphics[
            width=0.325\linewidth,
            height=0.20\textheight,
            keepaspectratio
        ]{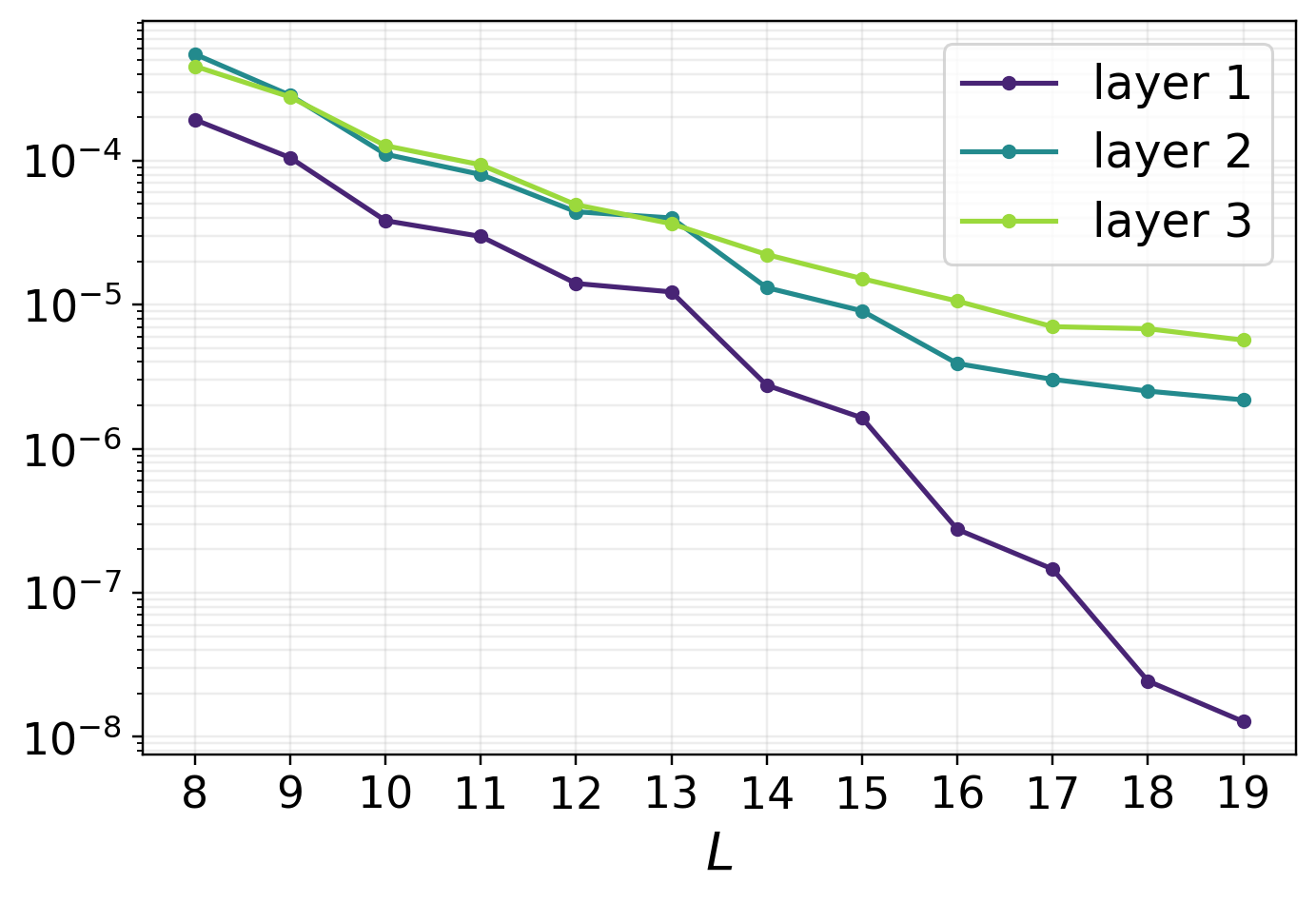}
        \hfill
        \includegraphics[
            width=0.325\linewidth,
            height=0.20\textheight,
            keepaspectratio
        ]{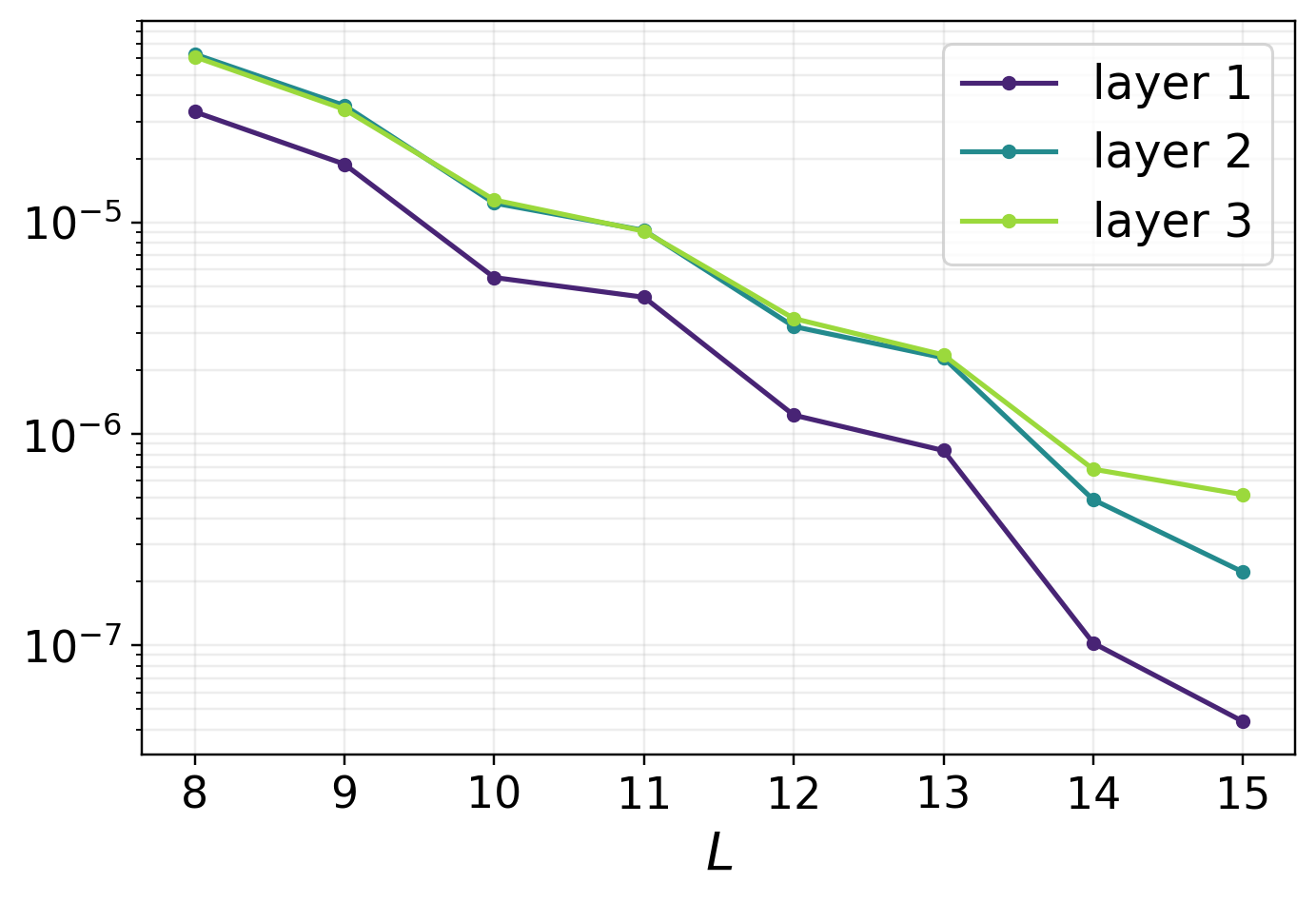}

        \caption{Sym8: \(F(x)=\exp(x)\)}
        \label{fig:sym8-exp}
    \end{subfigure}

    \caption{Mean Frobenius covariance errors versus \(L\) at
    layers \(\ell=1,2,3\). Columns correspond to
    \(\nu=0.5,1.5,2\).}
    \label{fig:combined-wavelets}

\end{figure}

\section*{Acknowledgments}
The general proof strategy is our own; some intermediate steps and simplifications were developed with the assistance of an OpenAI language model. The authors have verified all such content and take full responsibility for the correctness of the final manuscript.
\clearpage
\bibliographystyle{siamplain}
\bibliography{references}

\appendix
\section{Proof of Lemma~\ref{lem:kernel_properties}}\label{app:kernel-properties-proof}\label{app:gaussian-representation}\label{prop:matern-gaussian-mixture}\label{lem:conv-gaussian-feature-factorisation} The five subsections prove parts~(i)--(v), respectively. We first introduce the Gaussian representation used throughout the proof.

For \(\nu>0\), setting the correlation-length parameter to \(1\) in the Gaussian mixture representation of \cite[Theorem~1]{TronarpKarvonenSarkka2018} and substituting \(a=(2u)^{-1}\) gives the mixing measure
\[
    d\mu_{\nu}(a)
    =
    \frac{\nu^\nu}{2^\nu \Gamma(\nu)}
    a^{-\nu-1}
    \exp\!\left(-\frac{\nu}{2 a}\right)\,da .
\]
The substitution \(t=\nu/(2a)\) gives
\[
    \int_0^\infty d\mu_\nu(a)
    =\frac{1}{\Gamma(\nu)}
      \int_0^\infty t^{\nu-1}e^{-t}\,dt
    =1.
\]
Thus \(\mu_\nu\) is a probability measure with a strictly positive density and full support on \((0,\infty)\).
At unit correlation length, the Mat\'ern radial profile defined in~\eqref{eq:maternradial} satisfies
\begin{equation}\label{eq:matern-profile-mixture}
    \phi_{\nu,1}(s)
    =\sigma^2\int_0^\infty e^{-as^2}\,d\mu_\nu(a),
    \qquad s\ge0.
\end{equation}
Let \(\lambda\) satisfy \eqref{eq:scale-field-bounds}. For \(a>0\), define the Gaussian features
\begin{equation}\label{eq:conv-feature-map}
    \Phi_a(x,z)
    :=\left(\frac{4a}{\pi}\right)^{n/4}
      \lambda(x)^{-n/4}
      \exp\!\left(-\frac{2a\|z-x\|_2^2}{\lambda(x)}\right).
\end{equation}
Following the Gaussian convolution construction of
\cite{Paciorek2003}, set
\begin{equation}\label{eq:conv-gaussian-factorisation}
    k_a(x,y)
    :=\int_{\R^n}\Phi_a(x,z)\Phi_a(y,z)\,dz.
\end{equation}
The features are nonnegative and satisfy \(\|\Phi_a(x,\cdot)\|_{L^2(\R^n)}=1\), so Cauchy--Schwarz gives \(0\le k_a\le1\). Completing the square gives 
\begin{align*}
    \frac{\|z-x\|_2^2}{\lambda(x)} + \frac{\|z-y\|_2^2}{\lambda(y)} = \frac{\|x-y\|_2^2}{\lambda(x)+\lambda(y)} + \frac{\lambda(x)+\lambda(y)}{\lambda(x)\lambda(y)}\Big\|z-\frac{\lambda(y)x+\lambda(x)y}{\lambda(x)+\lambda(y)}\Big\|^2
\end{align*}
and thus shows
\[
    k_a(x,y)
    =\frac{2^{n/2}\lambda(x)^{n/4}\lambda(y)^{n/4}}
           {(\lambda(x)+\lambda(y))^{n/2}}
      \exp\!\left(-\frac{2a\|x-y\|_2^2}{\lambda(x)+\lambda(y)}\right).
\]
Using \eqref{eq:matern-profile-mixture} with \(s=\sqrt{2}\|x-y\|_2/\sqrt{\lambda(x)+\lambda(y)}\) gives the feature representation
\begin{equation}\label{eq:conv-matern-feature-mixture}
    k_{\conv}^{\lambda,k_{\M(\nu)}}(x,y)
    =\sigma^2\int_0^\infty k_a(x,y)\,d\mu_\nu(a).
\end{equation}
The explicit formula for \(k_a\) also shows that \((x,y,a)\mapsto k_a(x,y)\) is jointly continuous on \(\overline\Omega\times\overline\Omega\times(0,\infty)\).

Finally, the substitution \(t=\nu/(2a)\) gives, for \(0\le s<\nu\),
\begin{equation}\label{eq:matern-mixture-moments}
    \int_0^\infty a^s\,d\mu_\nu(a)
    =\left(\frac\nu2\right)^s\frac{\Gamma(\nu-s)}{\Gamma(\nu)}.
\end{equation}

\subsection{Weighted trace}
\label{app:gaussian-feature-sobolev}
We prove Lemma~\ref{lem:conv-matern-sobolev-trace} using two auxiliary estimates: a Sobolev product bound and an integrated Sobolev bound for the Gaussian features.

\begin{lemma}\label{lem:integer-sobolev-product}
Let \(\gamma_1,\dots,\gamma_s\in \N^n\) be nonzero multi-indices such that \(\sum_{j=1}^s |\gamma_j|\le r\).
Then
\[
    \left\|
        \prod_{j=1}^s \partial^{\gamma_j}u
    \right\|_{L^2(\Omega)}
    \le
    C_{\Omega,n,r,s}\|u\|_{H^r(\Omega)}^s,
    \qquad u\in H^r(\Omega).
\]
The case \(s=0\) is interpreted as the empty product \(1\).
\end{lemma}

\begin{proof}
    Set \(m_j:=|\gamma_j|\) and \(p_j:=2r/m_j\).  Since
\(r>n/2\) (Assumption~\ref{eq:fixed-r-assumption}) and $r\geq m_j$ by definition, we have
\[
    r-m_j\ge n\!\left(\frac12-\frac1{p_j}\right)\quad\implies\quad H^r(\Omega)\hookrightarrow W^{m_j,p_j}(\Omega).    
\]
Moreover
\(\sum_j(p_j/2)^{-1}=\sum_jm_j/r\le1\). H\"older's inequality for $\int_\Omega\prod_{j=1}^s (\partial^{\gamma_j}u)^2\,dx$ gives the result.
\end{proof}

\begin{lemma}
\label{lem:conv-gaussian-feature-sobolev}
Let \(\lambda\in C^\infty(\overline\Omega)\), and suppose \(
    0<\underline\lambda\le\lambda(x)\), \(x\in\overline\Omega\).
For \(a>0\), let \(\Phi_a\) be defined by
\eqref{eq:conv-feature-map}. Then
\begin{equation}\label{eq:conv-gaussian-feature-Hr}
    \int_{\R^n}
    \|\Phi_a(\cdot,z)\|_{H^r(\Omega)}^2\,dz
    \le
    C_{\Omega,n,r}
    (1+\underline\lambda^{-1})^{2r}
    (1+\|\lambda\|_{H^r(\Omega)})^{2r}
    \sum_{q=0}^{r}a^q.
\end{equation}
\end{lemma}

\begin{proof}
For fixed $x$, set
\[
    \rho(x):=\sqrt{\frac{a}{\lambda(x)}}, \qquad w:=2\rho(x)(z-x), \qquad H(w):=\frac{1}{2}\lVert w\rVert_2^2-\frac{n}{4}.
\]
Then
\[
    \Phi_a(x,z) = \left(\frac{4a}{\pi}\right)^{n/4} \lambda(x)^{-n/4}e^{-\lVert w\rVert_2^2/2}.
\]
The change of variables $z\mapsto w$ gives the exact identity
\begin{equation}\label{eq:change-of-vars}
    |\Phi_a(x,z)|^2\,dz = \pi^{-n/2}e^{-\lVert w\rVert_2^2}\,dw.
\end{equation}
Hence, the integral in $z$ of $|\Phi_a|^2$ times the square of any fixed polynomial in $w$ is a finite Gaussian moment, independent of $a$, $x$, and $\lambda$. We next describe the structure of the $x$-derivatives of $\Phi_a$. Direct differentiation gives, for $i=1,\dots,n$,
\begin{subequations}
\begin{align}
    \partial_{x_i}\rho &= -\frac{1}{2}\rho\, \frac{\partial_{x_i}\lambda}{\lambda}, \label{eq:conv-gaussian-drho}\\
    \partial_{x_i}w &= -2\rho e_i -\frac{1}{2}w\frac{\partial_{x_i}\lambda}{\lambda}, \label{eq:conv-gaussian-dw}\\
    \partial_{x_i}\Phi_a &= \Phi_a\left( 2\rho w_i +H(w)\frac{\partial_{x_i}\lambda}{\lambda} \right). \label{eq:conv-gaussian-dphi}
\end{align}
\end{subequations}
We claim that, for every multi-index $\alpha$ with $|\alpha|=m\le r$, $\partial_x^\alpha\Phi_a$ is a finite linear combination of terms of the form
\begin{equation}\label{eq:term-form}
    \Phi_a(x,z)\, \rho(x)^q P(w) \prod_{j=1}^{s} \frac{\partial^{\gamma_j}\lambda(x)}{\lambda(x)},
\end{equation} where $P$ is a polynomial, $\gamma_1,\dots,\gamma_s$ are nonzero multi-indices, and
\begin{equation}\label{eq:index-budget}
    q+\sum_{j=1}^{s}|\gamma_j|=m, \qquad 0\le q\le m.
\end{equation}
Moreover, for \(m\le r\), the number of terms, the coefficients of the linear combination, and the coefficients of the polynomials \(P\) are bounded in terms of \(n\) and \(r\) only. Each polynomial has degree at most \(2m\).

We prove the claim by induction on $m$. For $m=0$, it is immediate, with $q=0$, $s=0$, and $P=1$.
Suppose the claim holds at order $m<r$, and consider one term
\[
    T
    :=
    \Phi_a\,\rho^q P(w)
    \prod_{j=1}^{s}
    \frac{\partial^{\gamma_j}\lambda}{\lambda},\qquad
    \text{ where }\qquad
    q+\sum_{j=1}^{s}|\gamma_j|=m.
\]
We differentiate $T$ with respect to $x_i$, which leads to four cases.

\medskip
\noindent
\emph{Case 1: the derivative falls on $\Phi_a$.}
Using \eqref{eq:conv-gaussian-dphi},
\begin{align*}
    (\partial_{x_i}\Phi_a)\rho^qP(w)
    \prod_{j=1}^{s}\frac{\partial^{\gamma_j}\lambda}{\lambda}=
    2\Phi_a\rho^{q+1}
    &\bigl(w_iP(w)\bigr)
    \prod_{j=1}^{s}\frac{\partial^{\gamma_j}\lambda}{\lambda}\\&+
    \Phi_a\rho^q
    \bigl(H(w)P(w)\bigr)
    \frac{\partial_{x_i}\lambda}{\lambda}
    \prod_{j=1}^{s}\frac{\partial^{\gamma_j}\lambda}{\lambda}.
\end{align*}
In the first term, $q$ is replaced by $q+1$, while in the
second term a new factor corresponding to the multi-index $e_i$
is introduced. Hence each resulting term satisfies \(q+\sum_j|\gamma_j|=m+1\).

\medskip
\noindent
\emph{Case 2: the derivative falls on $\rho^q$.}
By \eqref{eq:conv-gaussian-drho},
\[
    \partial_{x_i}(\rho^q)
    =
    q\rho^{q-1}\partial_{x_i}\rho
    =
    -\frac{q}{2}\rho^q
    \frac{\partial_{x_i}\lambda}{\lambda}.
\]
Thus the resulting term is
\[
    -\frac{q}{2}
    \Phi_a\rho^qP(w)
    \frac{\partial_{x_i}\lambda}{\lambda}
    \prod_{j=1}^{s}
    \frac{\partial^{\gamma_j}\lambda}{\lambda}.
\]
Here $q$ is unchanged, while a new first-order derivative of
$\lambda$ is introduced. Thus \(q+\sum_j|\gamma_j|=m+1\) for the resulting term.

\medskip
\noindent
\emph{Case 3: the derivative falls on $P(w)$.}
Using \eqref{eq:conv-gaussian-dw},
\[
    \partial_{x_i}P(w)
    =
    -2\rho\,\partial_{w_i}P(w)
    -\frac{1}{2}
    \frac{\partial_{x_i}\lambda}{\lambda}
    w\cdot\nabla_wP(w).
\]
Hence
\begin{align*}
    \Phi_a\rho^q\partial_{x_i}P(w)
    \prod_{j=1}^{s}\frac{\partial^{\gamma_j}\lambda}{\lambda}
    =
    -2\Phi_a\rho^{q+1}
    \partial_{w_i}&P(w)
    \prod_{j=1}^{s}\frac{\partial^{\gamma_j}\lambda}{\lambda}\\&-
    \frac{1}{2}
    \Phi_a\rho^q
    \bigl(w\cdot\nabla_wP(w)\bigr)
    \frac{\partial_{x_i}\lambda}{\lambda}
    \prod_{j=1}^{s}\frac{\partial^{\gamma_j}\lambda}{\lambda}.
\end{align*}
Again, the first term replaces $q$ by $q+1$, while the second
introduces a new first-order derivative of $\lambda$. Thus both terms satisfy \(q+\sum_j|\gamma_j|=m+1\).

\medskip
\noindent
\emph{Case 4: the derivative falls on one of the $\lambda$-factors.}
For $k\in\{1,\dots,s\}$,
\begin{equation}\label{eq:lambda-factor}
    \partial_{x_i}
    \left(
        \frac{\partial^{\gamma_k}\lambda}{\lambda}
    \right)
    =
    \frac{\partial^{\gamma_k+e_i}\lambda}{\lambda}
    -
    \frac{\partial^{\gamma_k}\lambda}{\lambda}
    \frac{\partial_{x_i}\lambda}{\lambda}.
\end{equation}
Therefore the corresponding contribution is a sum of terms of the
form
\[
    \Phi_a\rho^qP(w)
    \frac{\partial^{\gamma_k+e_i}\lambda}{\lambda}
    \prod_{j\ne k}
    \frac{\partial^{\gamma_j}\lambda}{\lambda},
    \qquad\text{ or }\qquad
    -\Phi_a\rho^qP(w)
    \frac{\partial^{\gamma_k}\lambda}{\lambda}
    \frac{\partial_{x_i}\lambda}{\lambda}
    \prod_{j\ne k}
    \frac{\partial^{\gamma_j}\lambda}{\lambda}.
\]
In the first term, $|\gamma_k|$ is increased by one; in the second,
a new multi-index $e_i$ is introduced. Thus the total derivative
order carried by the $\lambda$-factors increases by one in either
case.

It follows that every term arising from $\partial_{x_i}T$ again has
the form \eqref{eq:term-form}, and in every case
\[
    q+\sum_{j=1}^{s}|\gamma_j|=m+1.
\]
This proves \eqref{eq:index-budget} at order \(m+1\).

It remains to control the polynomial degree. If $\deg P\le 2m$,
then the polynomial factors arising in the four cases are obtained
from $P$ by multiplication by $w_i$, multiplication by $H(w)$,
differentiation with respect to $w_i$, or application of
$w\cdot\nabla_w$. Since $H$ has degree two,
\[
    \deg(w_iP)\le 2m+1,
    \qquad
    \deg(H P)\le 2m+2,
\]
while
\[
    \deg(\partial_{w_i}P)\le 2m-1,
    \qquad
    \deg(w\cdot\nabla_wP)\le 2m.
\]
Hence every polynomial arising at order \(m+1\) has degree at most \(2(m+1)\). Each differentiation produces at most a number of terms depending on \(n\) and \(r\), through polynomial operations whose coefficients are bounded in terms of \(n\) and \(r\). Since at most \(r\) differentiations are taken, the number of terms and all coefficients remain bounded in terms of \(n\) and \(r\). This completes the induction.

Since \(\rho^q=a^{q/2}\lambda^{-q/2}\), each term in an $m$-th derivative can therefore be written as
\begin{equation}\label{eq:term-rewritten}
     \Phi_a(x,z)\, a^{q/2}\lambda(x)^{-q/2-s} P(w) \prod_{j=1}^{s}\partial^{\gamma_j}\lambda(x),
\end{equation} with
\(q+\sum_{j=1}^{s}|\gamma_j|=m\).
Fix $|\alpha|=m\le r$. By the preceding derivative expansion~\eqref{eq:term-form}, Cauchy--Schwarz over the finite sum, followed by \eqref{eq:change-of-vars}, gives
\begin{equation}\label{eq:deriv-bound}
    \int_{\mathbb{R}^n} |\partial_x^\alpha\Phi_a(x,z)|^2\,dz \le C_{n,r} \sum a^q\lambda(x)^{-q-2s} \left| \prod_{j=1}^{s}\partial^{\gamma_j}\lambda(x) \right|^2,
\end{equation}
where the finite sum is over the terms in the derivative expansion. All Gaussian polynomial moments have been absorbed into $C_{n,r}$. For every such term, since each $\gamma_j$ is nonzero, \[ s \le \sum_{j=1}^{s}|\gamma_j| = m-q. \]
Hence, \(q+2s \le 2m-q \le 2r\), and since $\lambda(x)\ge\underline{\lambda}>0$ we see that \(\lambda(x)^{-q-2s} \le (1+\underline{\lambda}^{-1})^{2r}\).
Identity \eqref{eq:index-budget} also gives \(\sum_{j=1}^{s}|\gamma_j| = m-q \le r\), and hence
Lemma~\ref{lem:integer-sobolev-product} gives, for every term,
\[
\left\| \prod_{j=1}^{s}\partial^{\gamma_j}\lambda \right\|_{L^2(\Omega)}^2 \le C_{\Omega,n,r} \bigl(1+\|\lambda\|_{H^r(\Omega)}\bigr)^{2r},
\]
where the empty product for $s=0$ is interpreted as $1$. Integrating \eqref{eq:deriv-bound} over $x$ and using $0\le q\le m$, we obtain
\[
\int_{\mathbb{R}^n} \|\partial_x^\alpha\Phi_a(\cdot,z)\|_{L^2(\Omega)}^2\,dz \le C_{\Omega,n,r} (1+\underline{\lambda}^{-1})^{2r} \bigl(1+\|\lambda\|_{H^r(\Omega)}\bigr)^{2r} \sum_{q=0}^{m}a^q.
\]
Summing over \(|\alpha|\le r\) proves \eqref{eq:conv-gaussian-feature-Hr}.
\end{proof}
\begin{proof}[Proof of Lemma~\ref{lem:conv-matern-sobolev-trace}]
Assume first that \(\lambda\in C^\infty(\overline\Omega)\), and fix \(\iota\in\Lambda_J\). Since \(\Phi_a\ge0\), Tonelli's theorem and \eqref{eq:conv-gaussian-factorisation} give
\begin{align*}
    \int_0^\infty
    \int_{\R^n}
    \iint_{\Omega\times\Omega}
    \Phi_a(x,z)\Phi_a(y,z)
    &|\psi_\iota(x)\psi_\iota(y)|
    \,dx\,dy\,dz\,d\mu_\nu(a)
    \\&=
    \int_0^\infty
    \iint_{\Omega\times\Omega}
    k_a(x,y)
    |\psi_\iota(x)\psi_\iota(y)|
    \,dx\,dy\,d\mu_\nu(a),
\end{align*}
which is at most \(\|\psi_\iota\|_{L^1(\Omega)}^2<\infty\), since \(0\le k_a\le1\), \(\mu_\nu\) is a probability measure, and \(\psi_\iota\in L^2(\Omega)\) with \(\Omega\) bounded. Fubini's theorem therefore applies, and the \(x\)- and \(y\)-integrals factor:
\[
    [\bC_J]_{\iota,\iota}
    =
    \sigma^2
    \int_0^\infty
    \int_{\R^n}
    \bigl|
        \langle
            \Phi_a(\cdot,z),
            \psi_\iota
        \rangle_{L^2(\Omega)}
    \bigr|^2
    dz\,d\mu_\nu(a).
\]
Since \(\Lambda_J\) is finite, the sum over \(\Lambda_J\) with weights \(2^{2rj(\iota)}\) may be taken inside the integrals. For each \(a>0\) and \(z\in\R^n\) the function
\(x\mapsto\Phi_a(x,z)\) lies in \(C^\infty(\overline\Omega)\subset H^r(\Omega)\),
since \(\lambda\) is smooth and bounded below, so
\eqref{eq:wavelet-analysis-Hr} may be applied with \(v=\Phi_a(\cdot,z)\),
giving
\[
\begin{aligned}
    \tr
    \bigl(
        D_{r,J}
        \bC_J
        D_{r,J}
    \bigr)
    &=
    \sigma^2
    \int_0^\infty
    \int_{\R^n}
    \sum_{\iota\in\Lambda_J}
    2^{2rj(\iota)}
    \bigl|
        \langle
            \Phi_a(\cdot,z),
            \psi_\iota
        \rangle
    \bigr|^2
    dz\,d\mu_\nu(a)
    \\
    &\le
    A_r\sigma^2
    \int_0^\infty
    \int_{\R^n}
    \|\Phi_a(\cdot,z)\|_{H^r(\Omega)}^2
    \,dz\,d\mu_\nu(a).
\end{aligned}
\]
Combining \eqref{eq:conv-gaussian-feature-Hr} with the moment identity \eqref{eq:matern-mixture-moments} proves \eqref{eq:conv-matern-weighted-trace-direct} for smooth \(\lambda\).

Using a bounded Sobolev extension
operator followed by mollification \cite{RAAdams_JJFFournier_2003}, choose
\(\lambda_m\in C^\infty(\overline\Omega)\) with \(\lambda_m\to\lambda\) in
\(H^r(\Omega)\). Since \(r>n/2\), the convergence is uniform on
\(\overline\Omega\), so after discarding finitely many indices
\[
    \lambda_m(x)\ge\tfrac12\underline\lambda,
    \quad x\in\overline\Omega,
    \qquad
    \|\lambda_m\|_{H^r(\Omega)}
    \le
    1+\|\lambda\|_{H^r(\Omega)}.
\]
Since \((1+2\underline\lambda^{-1})^{2r}\le2^{2r}(1+\underline\lambda^{-1})^{2r}\) and \((1+\|\lambda_m\|_{H^r(\Omega)})^{2r}\le2^{2r}(1+\|\lambda\|_{H^r(\Omega)})^{2r}\), the smooth case gives \eqref{eq:conv-matern-weighted-trace-direct} for every \(\lambda_m\), with the right-hand side evaluated at \(\lambda\) and a constant independent of \(m\) and \(J\).

By \eqref{eq:conv-matern-feature-mixture}, dominated convergence gives \(k_{\conv}^{\lambda_m,k_{\M(\nu)}}(x,y) \to k_{\conv}^{\lambda,k_{\M(\nu)}}(x,y)\) for every \(x,y\). All these kernels are bounded by \(\sigma^2\). Since \(\psi_\iota\otimes\psi_{\iota'}\in L^1(\Omega\times\Omega)\), dominated convergence also gives convergence of each covariance-matrix entry at every fixed \(J\).
For each fixed \(J\), the weighted traces converge to \(\tr(D_{r,J}\bC_JD_{r,J})\), so the bound passes to the limit with the same constant. Since this constant is independent of \(J\), taking the supremum over \(J\) proves \eqref{eq:conv-matern-weighted-trace-direct}.
\end{proof}

\subsection{Uniform traces and truncation bounds}\label{app:composition-sobolev}
\begin{proof}[Proof of Lemma~\ref{lem:dgp-scale-field-bounds}]
For \(R\ge0\), define
\[
    M_{F,r}(R)
    :=
    \max_{0\le k\le r}
    \sup_{|t|\le R}|F^{(k)}(t)|,
\]
which is finite and nondecreasing in \(R\) since \(F\in C^r(\R)\).
For smooth \(g\), repeated use of the chain and product rules gives, for
every \(1\le|\alpha|\le r\), (this is a variant of Fa\`a di Bruno's formula)
\[
    \partial^\alpha(F(g))
    =
    \sum_{k=1}^{|\alpha|}
    F^{(k)}(g)
    \sum_{\substack{
        \gamma_1+\cdots+\gamma_k=\alpha\\
        |\gamma_j|\ge1
    }}
    c_{\alpha;\gamma_1,\dots,\gamma_k}
    \prod_{j=1}^{k}\partial^{\gamma_j}g.
\]
In every summand,
\(\sum_j|\gamma_j|=|\alpha|\le r\).  Hence Lemma~\ref{lem:integer-sobolev-product},
together with Sobolev embedding, yields
\[
    \|F(g)\|_{H^r(\Omega)}
    \le
    C_{\Omega,n,r}
    M_{F,r}\!\left(
        C^{\mathrm{emb}}_{\Omega,r}\|g\|_{H^r}
    \right)
    (1+\|g\|_{H^r})^r.
\]
For general \(g\in H^r(\Omega)\), choose \(g_m\in C^\infty(\overline\Omega)\)
such that \(g_m\to g\) in \(H^r(\Omega)\) similarly to the previous proof and show \(F(g_m)\to F(g)\) in \(L^2(\Omega)\). By weak compactness, a subsequence converges weakly in \(H^r(\Omega)\), with limit \(F(g)\) by the uniqueness of limits. Weak lower semicontinuity of norms proves
\eqref{eq:dgp-composition-bound}.

If \(g\in\mathcal B_r(B)\), then
\[
    |g(x)|
    \le C^{\mathrm{emb}}_{\Omega,r}B
    \qquad
    \text{for every }x\in\overline\Omega.
\]
Since \(F\) is continuous and strictly positive, its minimum and
maximum on this compact interval are finite and the minimum is
strictly positive.  
For \(g\in\mathcal B_r(B)\), applying Lemma~\ref{lem:conv-matern-sobolev-trace} with \(\lambda=F(g)\), using \eqref{eq:dgp-composition-bound} and \(F(g)\ge\underline\lambda(B)\), gives \eqref{eq:dgp-trace-majorant-explicit}. The initial stationary Mat\'ern covariance is the case \(\lambda\equiv\lambda_0^2\). Finally, Theorem~\ref{thm:trace-controlled-scale-tail}, together with \eqref{eq:dgp-trace-majorant}, gives \eqref{eq:dgp-uniform-tail-bound}.
\end{proof}

\subsection{Covariance stability}
We first establish a log-scale Lipschitz estimate for the kernel, then deduce Lemma~\ref{lem:dgp-L2-covariance-lipschitz}.
\label{app:log-lipschitz}
\begin{lemma}\label{lem:conv-matern-log-scale-lipschitz}
Let \(\nu>1\), and for \(a,b>0\) and \(z\in\R^n\) define
\[
\kloc_\nu(a,b,z)
    :=
\frac{2^{n/2}a^{n/4}b^{n/4}}{(a+b)^{n/2}}
\phi_{\nu,1}\!\left(
\frac{\sqrt{2}\,\|z\|_2}{\sqrt{a+b}}
    \right),
\]
so that \(k_\ell(x,y;g)=\kloc_{\nu_\ell}\bigl(F(g(x)),F(g(y)),x-y\bigr)\)
for the convolution Mat\'ern kernel \eqref{eq:convolution-kernel}. With \(C_{\log}
    :=
\sigma^2\left(\frac n4+\frac1e\right)\), we have, for all \(a_1,a_2,b_1,b_2>0\) and \(z\in\R^n\),
\begin{equation}\label{eq:conv-matern-log-scale-lipschitz}
|\kloc_\nu(a_1,b_1,z)-\kloc_\nu(a_2,b_2,z)|
\le
C_{\log}
\left(
    |\log a_1-\log a_2|
+|\log b_1-\log b_2|
\right).
\end{equation}
\end{lemma}
\begin{proof}
Since \(\nu>1\), \eqref{eq:matern-mixture-moments} gives \(\int_0^\infty q\,d\mu_\nu(q)<\infty\). Dominated convergence therefore justifies differentiation of \eqref{eq:matern-profile-mixture} under the integral sign on compact \(s\)-intervals, giving
\[
\phi_{\nu,1}'(s)
    =
-2\sigma^2s\int_0^\infty qe^{-qs^2}\,d\mu_\nu(q).
\]
Since the integrand defining \(\phi_{\nu,1}\) is bounded by \(1\), and since
\(xe^{-x}\le e^{-1}\) for \(x\ge0\),
\begin{equation}\label{eq:matern-profile-log-derivative-bound}
    |\phi_{\nu,1}(s)|\le\sigma^2,
\qquad
    s|\phi_{\nu,1}'(s)|
    =
2\sigma^2\int_0^\infty qs^2e^{-qs^2}\,d\mu_\nu(q)
\le
\frac{2\sigma^2}e.
\end{equation}
Now set
\[
\theta:=\frac{a}{a+b}\in(0,1),
\qquad
R:=\frac{\sqrt2\,\|z\|_2}{\sqrt{a+b}},
\qquad
P:=\frac{2^{n/2}a^{n/4}b^{n/4}}{(a+b)^{n/2}},
\]
so that \(\kloc_\nu(a,b,z)=P\,\phi_{\nu,1}(R)\), and \(0<P\le1\) by the
arithmetic--geometric mean inequality. 
Since \(a\,\partial_aP=\frac n4(1-2\theta)P\) and
\(a\,\partial_aR=-\frac\theta2R\), differentiation gives
\[
    a\,\partial_a\kloc_\nu
    =
    P\left[
        \frac n4(1-2\theta)\phi_{\nu,1}(R)
        -\frac\theta2 R\phi_{\nu,1}'(R)
    \right].
\]
Using \(P\le1\), \(|1-2\theta|\le1\), \(\theta\le1\), together with the two estimates in \eqref{eq:matern-profile-log-derivative-bound}, gives
\[
\sup_{a,b>0,\,z\in\R^n}
    |a\,\partial_a\kloc_\nu(a,b,z)|
\le
\frac n4\sigma^2+\frac1e\sigma^2
    =C_{\log},
\]
and, by symmetry in \(a\) and \(b\), the same bound holds for
\(|b\,\partial_b\kloc_\nu|\). Neither bound in
\eqref{eq:matern-profile-log-derivative-bound} involves \(\nu\), so neither
does \(C_{\log}\).

For \eqref{eq:conv-matern-log-scale-lipschitz}, fix \(b,z\) and set \(h(u):=\kloc_\nu(e^u,b,z)\). The preceding estimate bounds \(h'(u)=e^u\partial_a\kloc_\nu(e^u,b,z)\) by \(C_{\log}\), so the fundamental theorem of calculus gives
\[
|\kloc_\nu(a_1,b,z)-\kloc_\nu(a_2,b,z)|
=\left|\int_{\log a_2}^{\log a_1}h'(u)\,du\right|
\le C_{\log}|\log a_1-\log a_2|,
\]
and likewise \(|\kloc_\nu(a_2,b_1,z)-\kloc_\nu(a_2,b_2,z)|\le
C_{\log}|\log b_1-\log b_2|\). Combining these estimates by the triangle inequality proves \eqref{eq:conv-matern-log-scale-lipschitz}.
\end{proof}

\begin{proof}[Proof of Lemma~\ref{lem:dgp-L2-covariance-lipschitz}]
By Assumption~\ref{as:dgp-fixed-regularity}, \(\nu_\ell>r\ge1\), so
Lemma~\ref{lem:conv-matern-log-scale-lipschitz} is available at layer
\(\ell\).  For \(x,y\in\Omega\), apply it with
\[
    a_1=F(f(x)),\quad b_1=F(f(y)),\qquad
    a_2=F(g(x)),\quad b_2=F(g(y)),
\]
all of which are finite and strictly positive because \(f,g\) are continuous
on \(\overline\Omega\) and \(F\) is strictly positive.
Since \(|\log F(s)-\log F(t)|\le L_F|s-t|\) by \eqref{eq:dgp-log-link-assumption},
\[
    |k_\ell(x,y;f)-k_\ell(x,y;g)|\le C_{\log}L_F\bigl(|f(x)-g(x)|+|f(y)-g(y)|\bigr).
\]
Using \((u+v)^2\le2(u^2+v^2)\) and integrating over \(\Omega\times\Omega\) gives
\[
    \|k_\ell(\cdot,\cdot;f)-k_\ell(\cdot,\cdot;g)\|_{L^2(\Omega^2)}
    \le
    2|\Omega|^{1/2}C_{\log}L_F
    \|f-g\|_{L^2(\Omega)}.
\]
Applying the analysis bound \eqref{eq:wavelet-analysis-L2} in each variable in turn yields
\[
    \sum_{\iota,\iota'\in\Lambda_J}
    |\langle k_\ell(f)-k_\ell(g),
      \psi_\iota\otimes\psi_{\iota'}\rangle|^2
    \le
    A_0^2\|k_\ell(f)-k_\ell(g)\|_{L^2(\Omega^2)}^2,
\]
which is the last inequality in \eqref{eq:dgp-global-frob-lipschitz}. The truncation only deletes entries, which gives the middle inequality, and Lemma~\ref{lem:psd-projection-nonexpansive} gives the first.
\end{proof}

\subsection{Strict positivity}
\label{app:strict-positivity}
\begin{proof}[Proof of Lemma~\ref{lem:conv-matern-integral-strict-positivity}]
Continuity and symmetry follow directly from \eqref{eq:convolution-kernel} and the
assumptions on \(\lambda\). 
Strict positive definiteness on finite sets of distinct points follows from \cite[Proposition~1]{dunlop2018deep} or~\cite[Theorem~1]{Paciorek2003}, with \(\Sigma(x)=\lambda(x)I_n\).
It remains to prove the strict integral inequality.  Let
\(u\in L^2(\Omega)\).
Using the feature representation \eqref{eq:conv-gaussian-factorisation}, for every \(a>0\), we get
\[
    Q_a(u)
    :=
    \int_\Omega\int_\Omega
    k_a(x,y)u(x)u(y)\,dx\,dy
    =
    \int_{\R^n}
    \left|
        \int_\Omega u(x)\Phi_a(x,z)\,dx
    \right|^2dz
    \ge0.
\]
  Moreover, \(a\mapsto Q_a(u)\) is continuous by dominated convergence, since \(k_a(x,y)\) is continuous in \(a\) and the integrands are dominated by \(|u(x)u(y)|\in L^1(\Omega\times\Omega)\).
The mixture representation \eqref{eq:conv-matern-feature-mixture} gives
\[
    Q(u)
    :=
    \int_\Omega\int_\Omega
    k_{\conv}^{\lambda,k_{\M(\nu)}}(x,y)u(x)u(y)\,dx\,dy
    =
    \sigma^2\int_0^\infty Q_a(u)\,d\mu_\nu(a)
    \ge0.
\]
Suppose now that \(Q(u)=0\). Since \(Q_a(u)\) is nonnegative and continuous in \(a\), and \(\mu_\nu\) has full support, \(Q_a(u)=0\) for every \(a>0\). Set
\[
    F_a(z):=\int_\Omega u(x)\Phi_a(x,z)\,dx,\qquad a>0.
\]
Then \(F_a=0\) in \(L^2(\R^n)\).  Also \(F_a\in L^1(\R^n)\), because
the Gaussian integral gives
\begin{align*}
    \int_{\R^n}|F_a(z)|\,dz&\le
    \int_\Omega|u(x)|
    \int_{\R^n}\Phi_a(x,z)\,dz\,dx\\&=
    \int_\Omega|u(x)|
    \left(\frac{\pi\lambda(x)}{a}\right)^{n/4}dx\le
    \left(\frac{\pi\overline\lambda}{a}\right)^{n/4}
    \|u\|_{L^1(\Omega)}<\infty.
\end{align*}
Its Fourier transform is therefore defined pointwise and has the form
\[
    \widehat F_a(\xi)
    =
    c_{a,n}
    \int_\Omega
    u(x)\lambda(x)^{n/4}e^{-i\xi\cdot x}
    e^{-\lambda(x)\|\xi\|_2^2/(8a)}\,dx,
\]
where \(c_{a,n}>0\).  Since \(F_a=0\) almost everywhere,
\(\widehat F_a\equiv0\).  With
\(v(x):=u(x)\lambda(x)^{n/4}\), this gives
\begin{equation}\label{eq:strict-positivity-weighted-fourier-identity}
    \int_\Omega
    v(x)e^{-i\xi\cdot x}
    e^{-\lambda(x)\|\xi\|_2^2/(8a)}\,dx
    =0,
    \qquad \xi\in\R^n,\ a>0.
\end{equation}
For each fixed \(\xi\in\R^n\), let \(a\to\infty\) in \eqref{eq:strict-positivity-weighted-fourier-identity}. The integrands are bounded in modulus by \(|v|\in L^1(\Omega)\), so dominated convergence gives \(\widehat v(\xi)=0\), where \(v\) is extended by zero outside \(\Omega\). Uniqueness of the Fourier transform on \(L^1\) gives \(v=0\) almost everywhere; see \cite[Corollary~8.27]{Folland1999}. Since \(\lambda\) is bounded away from zero, \(u=0\) almost everywhere. Thus \(Q(u)>0\) for every \(0\ne u\in L^2(\Omega)\), as required.
\end{proof}

\subsection{Lower spectral bounds}\label{app:lower-spectral-bounds}
\begin{proof}[Proof of Lemma~\ref{lem:dgp-finite-resolution-coercivity}]
The first assertion follows from Lemma~\ref{lem:conv-matern-integral-strict-positivity} with the constant function \(\lambda\equiv\lambda_0^2\), since the finite primal wavelet family is linearly independent by \eqref{eq:wavelet-biorthogonality}. For the second, suppose \(\kappa_{\ell,J}(B)=0\). Choose \(g_q\in\mathcal B_r(B)\) and unit vectors \(a_q\) such that
\(a_q^\top\CellJ{\ell}(g_q)a_q\to0\). After passing to subsequences,
\[
    g_q\rightharpoonup g_*\quad\text{in }H^r(\Omega),
    \qquad
    g_q\to g_*\quad\text{in }C(\overline\Omega),
    \qquad
    a_q\to a_*,\quad\|a_*\|_2=1,
\]
where the first two convergences use \(r>n/2\) and the compact Sobolev embedding
\cite{RAAdams_JJFFournier_2003}. Weak lower semicontinuity gives \(g_*\in\mathcal B_r(B)\). By Lemma~\ref{lem:dgp-L2-covariance-lipschitz},
\[
    \|\CellJ{\ell}(g_q)-\CellJ{\ell}(g_*)\|_{\op}\le\Lip_\ell\|g_q-g_*\|_{L^2(\Omega)}\to0,
\]
so \(a_*^\top\CellJ{\ell}(g_*)a_*=\lim_q a_q^\top\CellJ{\ell}(g_q)a_q=0\). Since \(F(g_*)\) is continuous and bounded above and away from zero, Lemma~\ref{lem:conv-matern-integral-strict-positivity} makes
\(\CellJ{\ell}(g_*)\) positive definite, a contradiction.
\end{proof}

\end{document}

%% file: ex_shared.tex
\usepackage{lipsum}
\usepackage{amsfonts}
\usepackage{graphicx}
\usepackage{epstopdf}
\usepackage{algorithmic}
\usepackage{amssymb}
\usepackage{xcolor}
\usepackage{mathrsfs}
\usepackage{subcaption}
\usepackage{enumitem}
\usepackage{microtype}
\usepackage[T1]{fontenc}

\newcommand{\R}{\mathbb{R}}
\newcommand{\E}{\mathbb{E}}
\newcommand{\Prob}{\mathbb{P}}
\newcommand{\N}{\mathbb{N}}
\newcommand{\bC}{\mathbf{C}}
\newcommand{\Cop}{\mathscr{C}}
\newcommand{\op}{\mathrm{op}}
\newcommand{\kloc}{k^{\mathrm{loc}}}
\newcommand{\nnz}{\mathrm{nnz}}

\newcommand{\CellJ}[1]{\bC_{#1}}
\newcommand{\CellJL}[1]{\bC_{#1}^{L}}

\DeclareMathOperator{\conv}{conv}
\DeclareMathOperator{\M}{Mat}
\DeclareMathOperator{\tr}{tr}
\DeclareMathOperator{\Lip}{Lip}
\ifpdf
  \DeclareGraphicsExtensions{.eps,.pdf,.png,.jpg}
\else
  \DeclareGraphicsExtensions{.eps}
\fi

\newsiamremark{remark}{Remark}
\newsiamremark{hypothesis}{Hypothesis}
\crefname{hypothesis}{Hypothesis}{Hypotheses}
\newsiamthm{claim}{Claim}
\newsiamthm{assumption}{Assumption}

\headers{Wavelet compression for sampling deep Gaussian processes}{Oskar Broukal, Michael Feischl, Conor Osborne}

\title{Wavelet compression for sampling deep Gaussian processes\thanks{Submitted to the editors 10/2026.
\funding{Funded by the Deutsche Forschungsgemeinschaft (DFG, German Research Foundation) -- Project-ID 258734477 -- SFB 1173, the Austrian Science Fund (FWF)
under the special research program Taming complexity in PDE systems (grant SFB F65) as well as projects I6667-N and P36150. This work was also supported by the ERC grant OPTIMAL (grant agreement No. 101125225).}}}

\author{Oskar Broukal\thanks{TU Wien, 
  \email{\{oskar.broukal, michael.feischl, conor.osborne\}@asc.tuwien.ac.at}}
\and Michael Feischl\footnotemark[2]
\and Conor Osborne\footnotemark[2]}

\usepackage{amsopn}
\DeclareMathOperator{\diag}{diag}

\makeatletter
\newcommand*{\addFileDependency}[1]{
  \typeout{(#1)}
  \@addtofilelist{#1}
  \IfFileExists{#1}{}{\typeout{No file #1.}}
}
\makeatother
